\documentclass[11pt,reqno]{amsart}

\newtheorem{proposition}{Proposition}[section]

\newtheorem{corollary}[proposition]{Corollary}
\newtheorem{theorem}[proposition]{Theorem}

\theoremstyle{definition}
\newtheorem{definition}[proposition]{Definition}
\newtheorem{example}[proposition]{Example}
\newtheorem{examples}[proposition]{Examples}
\newtheorem{remark}[proposition]{Remark}
\newtheorem{remarks}[proposition]{Remarks}

\newcommand{\thlabel}[1]{\label{th:#1}}
\newcommand{\thref}[1]{Theorem~\ref{th:#1}}
\newcommand{\selabel}[1]{\label{se:#1}}
\newcommand{\seref}[1]{Section~\ref{se:#1}}

\newcommand{\prlabel}[1]{\label{pr:#1}}
\newcommand{\prref}[1]{Proposition~\ref{pr:#1}}
\newcommand{\colabel}[1]{\label{co:#1}}
\newcommand{\coref}[1]{Corollary~\ref{co:#1}}
\newcommand{\relabel}[1]{\label{re:#1}}
\newcommand{\reref}[1]{Remark~\ref{re:#1}}
\newcommand{\exlabel}[1]{\label{ex:#1}}
\newcommand{\exref}[1]{Example~\ref{ex:#1}}
\newcommand{\delabel}[1]{\label{de:#1}}
\newcommand{\deref}[1]{Definition~\ref{de:#1}}
\newcommand{\eqlabel}[1]{\label{eq:#1}}
\newcommand{\equref}[1]{(\ref{eq:#1})}

\def\ot{\otimes}

\def\RR{{\mathbb R}}

\def\NN{{\mathbb N}}

\newcommand{\Cc}{\mathcal{C}}

\newcommand{\Mm}{\mathcal{M}}

\def\*C{{}^*\hspace*{-1pt}{\Cc}}
\def\text#1{{\rm {\rm #1}}}

\input xy
\xyoption {all} \CompileMatrices

\usepackage{amssymb}
\usepackage{color,amssymb,graphicx,amscd,amsmath,dsfont,stmaryrd,xcolor,comment}
\usepackage[pagebackref,colorlinks,urlcolor=blue,linkcolor=blue,citecolor=blue]{hyperref}

\begin{document}

\title[Universal constructions for Poisson algebras]
{Universal constructions for Poisson algebras. Applications}

\author{A. L. Agore}
\address{Vrije Universiteit Brussel, Pleinlaan 2, B-1050 Brussels, Belgium}
\address{Max Planck Institut f\"{u}r Mathematik, Vivatsgasse 7, 53111 Bonn, Germany}
\address{Simion Stoilow Institute of Mathematics of the Romanian Academy, P.O. Box 1-764, 014700 Bucharest, Romania}
\email{ana.agore@vub.be and ana.agore@gmail.com}

\author{G. Militaru}
\address{Faculty of Mathematics and Computer Science, University of Bucharest, Str.
Academiei 14, RO-010014 Bucharest 1, Romania}
\address{Simion Stoilow Institute of Mathematics of the Romanian Academy, P.O. Box 1-764, 014700 Bucharest, Romania}
\email{gigel.militaru@fmi.unibuc.ro and gigel.militaru@gmail.com}

\thanks{This work was supported by a grant of the Ministry of Research, Innovation and Digitization, CNCS/CCCDI --
UEFISCDI, project number PN-III-P4-ID-PCE-2020-0458, within PNCDI III}


\maketitle

\begin{abstract}
We introduce the \emph{universal algebra} of two Poisson algebras $P$ and $Q$ as a commutative algebra $A:={\mathcal P} (P, \, Q )$ satisfying a certain universal property. The universal algebra is shown to exist for any finite dimensional Poisson algebra $P$ and several of its applications are highlighted. For any Poisson $P$-module $U$, we construct a functor $U \ot - \colon {}_{A} \Mm \to {}_Q{\mathcal P} \Mm$ from the category of $A$-modules to the category of Poisson $Q$-modules which has a left adjoint whenever $U$ is finite dimensional. Similarly, if $V$ is an $A$-module, then there exists another functor $ - \ot V \colon {}_P{\mathcal P} \Mm \to {}_Q{\mathcal P} \Mm$ connecting the categories of Poisson representations of $P$ and $Q$ and the latter functor also admits a left adjoint if $V$ is finite dimensional. If $P$ is $n$-dimensional, then ${\mathcal P} (P) := {\mathcal P} (P, \, P)$ is the initial object in the category of all commutative bialgebras coacting on $P$. As an algebra, ${\mathcal P} (P)$ can be deescribed as the quotient of the polynomial algebra $k[X_{ij} \, | \, i, j = 1, \cdots, n]$ through an ideal generated by $2 n^3$ non-homogeneous polynomials of degree $\leq 2$. Two applications are provided. The first one describes the automorphisms group ${\rm Aut}_{\rm Poiss} (P)$ as the group of all invertible group-like elements of the finite dual ${\mathcal P} (P)^{\rm o}$. Secondly, we show that for an abelian group $G$, all $G$-gradings on $P$ can be explicitly described and classified in terms of the universal coacting bialgebra ${\mathcal P} (P)$.
\end{abstract}

\section*{Introduction}
Introduced in Hamiltonian mechanics as the dual of the category of classical mechanical systems, Poisson algebras play an important role
in the study of quantum groups, differential geometry, noncommutative geometry, integrable systems, quantum field theory or vertex operator algebras (see \cite{crainic, dr, gra2013, Hue, Kon, LPV, vanB}). Poisson algebras can be thought of as the algebraic counterpart of Poisson manifolds which are smooth manifolds $M$ whose commutative algebra $C^{\infty} (M, \, \RR)$ of real smooth functions is endowed with a Lie bracket $[-, -]$ satisfying the Leibniz rule, i.e. $C^{\infty} (M, \, \RR)$ is a Poisson algebra.

In this paper we introduce and study some universal objects for Poisson algebras and highlight their main applications having as source of inspiration the previous work of Sweedler \cite{Sw}, Manin \cite{Manin} and Tambara \cite{Tambara} for Hopf algebra (co)actions on associative algebras. From a categorical point of view, the existence of universal objects with a certain property, for a given category $\mathcal{C}$ can shed some light on the structure of the category $\mathcal{C}$ itself. In particular, the existence and description of universal objects (groups or "group-like objects" such as Lie groups, algebraic groups, Hopf algebras, groupoids or quantum groupoids, etc.) which act or coact on a fixed object $\mathcal{O}$ in a certain category $\mathcal{C}$ has often various applications in many areas of mathematics. An elementary but illuminating example is the following: let $\mathcal{O}$ be a given object in a certain category $\mathcal{C}$ and consider the category ${\rm ActGr}_{\mathcal{O}}$ of all groups that \emph{act} on $\mathcal{O}$, i.e. the objects in ${\rm ActGr}_{\mathcal{O}}$ are pairs $(G, \, \varphi)$ consisting of a discrete group $G$ and a morphism of groups $\varphi : G \to {\rm Aut}_{\mathcal{C}} (\mathcal{O})$, where ${\rm Aut}_{\mathcal{C}} (\mathcal{O})$ denotes the automorphisms group of the object $\mathcal{O}$ in $\Cc$. Then the category ${\rm ActGr}_{\mathcal{O}}$ has a final object, namely $\bigl({\rm Aut}_{\mathcal{C}} (\mathcal{O}), \, {\rm Id} \bigl)$. Now, if we replace the discrete groups that act on the fixed object $\mathcal{O}$ in $\mathcal{C}$, by some other "group-like objects" from a certain more sophisticated category $\mathcal{D}$ (for instance, Lie groups, algebraic groups, Hopf algebras, etc.) which (co)act on $\mathcal{O}$ and if moreover we ask the (co)action to preserve the algebraic, differential or topological structures which might exist on $\mathcal{O}$, then things become very complicated. Indeed, the first obstacle we encounter is the fact that ${\rm Aut}_{\mathcal{C}} (\mathcal{O})$ might not be an object \emph{inside} the category $\mathcal{D}$ anymore. However, even in this complicated situation, it is possible for the above result to remain valid but, however, the construction of the final object will be far more complicated. Furthermore, it is to expect that, if it exists, this final object will contain important information on the entire automorphisms group of the object $\mathcal{O}$. To the best of our knowledge, the first result in this direction was proved by Sweelder \cite[Theorem 7.0.4]{Sw} in the case where $\mathcal{C}$ is the category of associative algebras and $\mathcal{D}$ is the category of bialgebras: if $A$ is a fixed associative algebra then the category of all bialgebras $H$ that \emph{act} on $A$ (i.e. $A$ is an $H$-module algebra) has a final object ${\rm M} (A)$, called by Sweedler the \emph{universal measuring bialgebra} of $A$. The dual situation of \emph{coactions} of bialgebras on a fixed algebra $A$, was first considered in the case when $\mathcal{C}$ is the category of graded algebras by Manin \cite{Manin} for reasons related to non-commutative geometry, and in the general case by Tambara \cite{Tambara}. If $A$ is an associative algebra, necessarily finite dimensional this time around, then the category of all bialgebras that coact on $A$ (i.e. $A$ is an $H$-comodule algebra) has an intial object $a(A)$. The results have been extended in recent years to the setting of bialgebroids coactions in \cite{Aless1, Aless2}. Furthermore, the usual automorphisms group ${\rm Aut}_{\rm Alg} (A)$ of $A$ is indeed recovered as the group of all invertible group-like elements of the finite dual $a (A)^{\rm o}$ \cite[Theorem 2.1]{mil} and $a (A)^{\rm o}$ is just Sweedler's final object in the category of all bialgebras that act on $A$ \cite[Remark 1.3]{Tambara}. The two results above remains valid if we take the category of Hopf algebras instead of bialgebras: in particular, the Hopf envelope of $a(A)$, denoted by $\underline{{\rm aut}} (A)$, is called in non-commutative geometry the \emph{non-commutative symmetry group} of $A$ \cite{theo} and its description is a very complicated matter. The existence and description of these universal (co)acting bialgebras/Hopf algebras has been considered recently in \cite{AGV2} in the context of $\Omega$-algebras. The duality between Sweedler's and Manin-Tambara's objects has been extended to this general setting and necessary and sufficient conditions for the existence of the universal coacting bialgebras/Hopf algebras, which roughly explains the need for assuming finite-dimensionality in Manin-Tambara's constructions, are given. 
Note that the existence of the universal algebra $A:={\mathcal P} (P, \, Q )$ of two Poisson algebras $P$ and $Q$  as introduced in the present paper (\deref{defcheie}) does not follow from \cite{AGV2} as $A$ is assumed to be commutative and universal only among commutative comeasuring algebras.

Furthermore, universal coacting objects for Poisson algebras have also been considered in \cite{aa2019} but from a different perspective, leading to entirely different constructions. We only point out that in \cite{aa2019}, the universal coacting object considered is actually a Poisson Hopf algebra while the one constructed in this paper is just a commutative algebra without any Lie structure on it. For more background on the importance and the applications of universal bialgebras/Hopf algebras in various areas of mathematics we refer to \cite{AGV1, Aless1, Aless2, alex1, alex2, HWWW, Huang}.

The paper is organized as follows. \deref{defcheie} introduces the key object of our work, namely the \emph{universal algebra} of two Poisson algebras $P$ and $Q$, as a pair $\bigl( {\mathcal P} (P, \, Q ), \, \eta \bigl)$ consisting of a commutative algebra $A:={\mathcal P} (P, \, Q )$ and a Poisson algebra homomorphism
$\eta \colon Q \to P \ot {\mathcal P} (P, \, Q)$ satisfying a certain universal property. \thref{adj} proves that if $P$ is finite dimensional, then the universal algebra ${\mathcal P} (P, \, Q )$ of $P$ and $Q$ exists and its explicit construction is provided. This result has two important consequences: as proved in \thref{tensor_f}, for a fixed
Poisson $P$-module $U$ there exists a canonical functor $U \ot - \colon {}_{A} \Mm \to {}_Q{\mathcal P} \Mm$ from the category of usual $A$-modules (i.e. representations of the associative algebra $A$) to the category of Poisson $Q$-modules (i.e. Poisson representations of $Q$) and moreover, if $U$ is finite dimensional this functor has a left adjoint (\thref{adjoint_tens}). Secondly, if $V$ is an $A$-module, then there exists a canonical functor $ - \ot V \colon {}_P{\mathcal P} \Mm \to {}_Q{\mathcal P} \Mm$ connecting the categories of Poisson modules over $P$ and $Q$ and, furthermore, if $V$ is finite dimensional then the aforementioned functor has a left adjoint (\thref{adjoint_tens_2}). These results provide answers, at the level of Poisson algebras, to the following general question: \emph{if $\mathcal{O}_1$ and $\mathcal{O}_2$ are two mathematical objects (not necessary in the same category), is it possible to construct "canonical functors" between the representation categories ${\rm Rep}(\mathcal{O}_1)$ and ${\rm Rep}(\mathcal{O}_2)$ of the two objects?}

In \seref{sect3} we consider three more applications of our constructions. For an $n$-dimensional Poisson algebra $P$, we denote ${\mathcal P} (P) := {\mathcal P} (P, \, P)$ and we construct ${\mathcal P} (P)$ as the quotient of the polynomial algebra $k[X_{ij} \, | \, i, j = 1, \cdots, n]$ through an ideal generated by $2 n^3$ non-homogeneous polynomials of degree $\leq 2$. ${\mathcal P} (P)$ has a canonical bialgebra structure and \thref{univbialg} shows that ${\mathcal P} (P)$ is the \emph{initial object} of the category ${\rm CoactBialg}_P$ of all commutative bialgebras coacting on $P$ and, for this reason, we call it the \emph{universal coacting bialgebra} of $P$. As in the case of Lie \cite{am20} or associative algebras \cite{mil}, the universal bialgebra ${\mathcal P} (P)$ has two important applications, which provide the theoretical answer for Poisson algebras, of the following open questions: \emph{(1) Describe explicitly the automorphisms group of a given Poisson algebra $P$; (2) Describe and classify all $G$-gradings on $P$ for a given abelian group $G$}. More precisely, \thref{automorf} proves that there exists an isomorphism of groups between the group of all Poisson automorphisms of $P$ and the group of all invertible group-like elements of the finite dual ${\mathcal P} (P)^{\rm o}$. The second application is given in \thref{nouaclas}: for an abelian group $G$, all $G$-gradings on a finite dimensional Poisson algebra $P$ are described and classified in terms of bialgebra homomorphisms ${\mathcal P} (P) \to k[G]$. By taking Takeuchi's commutative Hopf envelope of ${\mathcal P} (P)$, we obtain that the category ${\rm CoactHopf}_P$ of all commutative Hopf algebras coacting on $P$ has an initial object ${\mathcal H} (P)$ (\coref{unihopf}). It is reasonable to hope that ${\mathcal H} (P)$ will play the role of a non-commutative symmetry group of the Poisson algebra $P$. This expectation is based on the fact that the concept of Poisson $H$-comodule algebra which we are dealing with, is the algebraic counterpart of the action of an algebraic groups on an affine Poisson variety \cite[Example 2.20]{gue}.

\section{Preliminaries}\selabel{prel}
All vector spaces, (bi)linear maps, unadorned tensor products, associative, Lie or Poisson
algebras and so on are over an arbitrary field $k$. Throughout, $\delta_{s, 1}$
will stand for Kronecker's symbol. A Poisson algebra is a vector space $P$ which admits both an (non-necessarily unital) associative
commutative algebra and a Lie algebra such that for all $x$, $y$, $z
\in P$ we have:
\begin{equation}\eqlabel{1}
[x,\, yz] = [x,\,y]\,z + y\, [x,\, z].
\end{equation}
A morphism of two Poisson algebras $P_1$ and $P_2$ is a linear map
$f\colon P_{1} \to P_{2}$ which is both an algebra homomorphism as well as a Lie
algebra homomorphism; if $P_1$ and $P_2$ are unital Poisson algebras then a Poisson homomorphism will be assumed to preserve units. We denote by ${\rm Aut}_{{\rm Poiss}} (P)$ the automorphisms group of a Poisson algebra $P$.

Let $P$ be a Poisson algebra. A (left) \emph{Poisson $P$-module} \cite{am15, umb} is a vector space $V$ equipped with two bilinear
maps $\triangleright \colon P \times V \to V$ and $\rightharpoonup \colon P
\times V \to V$ such that $(V, \, \triangleright)$ is a left
$P$-module, $(V, \, \rightharpoonup)$ is a left Lie $P$-module
satisfying the following two compatibility conditions for all $a$,
$b\in P$ and $x\in V$:
\begin{eqnarray}
&& (ab) \rightharpoonup  x =  a \triangleright (b \rightharpoonup x) + b \triangleright (a \rightharpoonup x) \eqlabel{Pmod1}\\
&&[a,\, b] \triangleright x = a \rightharpoonup (b \triangleright x) - b \rightharpoonup (a \triangleright
x) \eqlabel{Pmod2}
\end{eqnarray}
We denote by ${}_P{\mathcal P} \Mm$ the category of Poisson $P$-modules having as morphisms all linear maps which are compatible with both actions.

\begin{remarks} \relabel{freeposs}
1. The category ${}_P{\mathcal P} \Mm$ of Poisson $P$-modules is equivalent to the category
of usual left $P^e$-modules (\cite[Corollary 1]{umb}), where $P^e$ is the universal enveloping algebra of $P$ as constructed there. In particular, for any set $S$
we denote by $(P^e)^{(S)}$, the free $P^e$-module generated by $S$, which is the free Poisson $P$-module generated by $S$. Any quotient $(P^e)^{(S)}/ N$
through a Poisson submodule $N$ generated by a system of generators $R$ is called the free Poisson $P$-module generated by $S$ and the relations $R$.

2. A \emph{representation} of a Poisson algebra $P$ on a vector space $V$ \cite[Remarks 2.9]{am15} is a pair $(\psi, \, \varphi)$ consisting
of an algebra map $\psi : P \to {\rm End}_k (V)$, a Lie algebra map $\varphi : P \to {\rm gl}_k(V)$ such that for any $a$,
$b\in P$:
$$
\varphi (ab) = \psi(a) \circ \varphi(b) + \psi(b) \circ \varphi(a), \quad
\psi ( [a, b] ) = \varphi(a) \circ \psi (b) - \varphi (b) \circ \psi (a).
$$
The concepts of a Poisson $P$-module structure on $V$ and a representation of $P$ on $V$ are obviously equivalent.
\end{remarks}

We shall denote by ${\rm Poiss}_k$, ${\rm Poiss}_k^1 $ and ${\rm
ComAlg}_k$ the categories of Poisson, unital Poisson and respectively
unital commutative associative algebras over $k$. Furthermore, the category of
commutative bialgebras (resp. Hopf algebras) is denoted by ${\rm ComBiAlg}_k$ (resp. ${\rm ComHopf}_k$). For a coalgebra $C$ we denote by $G(C)$ the set of group like elements of $C$, i.e.  $G(C):= \{x\in C \, | \, \Delta (x) = x \ot
x \,\, {\rm and } \,\, \varepsilon(x) = 1 \}$. If $B$ is a bialgebra, then $G(B)$ is a monoid with respect to the multiplication on $B$. Throughout, for a bialgebra $B$, we denote by $B^{\rm o}$ its finite dual. Recall that if $H$ and $L$ are two bialgebras then the abelian
group ${\rm Hom}_k \, (H, \, L)$ is an associative algebra under the convolution product \cite{Sw}:
$(\theta_1 \star \theta_2) (h) := \sum \, \theta_1 (h_{(1)})
\theta_2 (h_{(2)})$, for all $\theta_1$, $\theta_2 \in {\rm Hom}_k \,
(H, \, L)$ and $h\in H$.

If $H$ is a commutative bialgebra (or a Hopf algebra), then a Poisson algebra $P$ is called a \emph{right Poisson $H$-comodule algebra} \cite{bal} (we also say
that \emph{$H$ coacts} on $P$) if there exists $\rho_P : P \to P \ot H$ a Poisson algebra map (the Poisson algebra structures on
$P\ot H$ are given by \equref{2} below) that is also a right $H$-comodule stucture on $P$. If $(P, \, \rho_P)$ is a right Poisson $H$-comodule algebra, then
the subalgebra of coinvariants $P^{{\rm co} (H)} := \{p \in P \, | \, \rho_P (p) = p \ot 1_H\}$ is a Poisson subalgebra of $P$. For a fixed Poisson algebra $P$ we denote by ${\rm CoactBialg}_P$ (resp. ${\rm CoactHopf}_P$) the category of all commutative bialgebras (resp. Hopf algebras) coacting on $P$. That is, the objects are all pairs $(H, \rho_P)$ consisting of a commutative bialgebra (resp. Hopf algebra) $H$ together with a structure of a right Poisson $H$-comodule algebra $\rho_P \colon P \to P \ot H$ while morphisms $f \colon (H, \rho_P) \to (H', \rho'_P)$ in ${\rm CoactBialg}_P$ are bialgebra maps $f: H \to H'$ such
that $({\rm Id}_P \ot f ) \circ \rho_P = \rho'_P$.

\begin{examples} \exlabel{nouexnev}
1. The first basic example of a Poisson $H$-comodule algebra is the one induced  by $G$-graded Poisson algebras. Recall that,
given an abelian group $G$ and a Poisson algebra $P$, a \emph{$G$-grading} on $P$ is a vector space
decomposition $P = \oplus_{\sigma \in G} \,
P_{\sigma}$ such that $P_{\sigma} P_{\tau} \subseteq P_{\sigma \tau}$ and $\left [P_{\sigma}, \,
P_{\tau} \right] \subseteq P_{\sigma \tau}$, for all $\sigma$, $\tau \in G$. Two $G$-gradings $P = \oplus_{\sigma \in G}
\, P_{\sigma} = \oplus_{\sigma \in G} \,
P_{\sigma} ^{'}$ on $P$ are called
\emph{isomorphic} if there exists $w \in {\rm Aut}_{{\rm Poiss}}
(P)$ an automorphism of $P$ such that $w
(P_{\sigma}) = P_{\sigma} ^{'}$, for
all $\sigma \in G$. Let $k[G]$ be the group algebra of $G$. By extending a well known result in Hopf algebra theory
(\cite[Excercise 3.2.21]{radford}) one can easily see that there is a bijection between the set of all right Poisson $k[G]$-comodule
structures $\rho \colon P \to P \ot k[G]$ on the Poisson algebra
$P$ and the set of all $G$-gradings on $P = \oplus_{\sigma \in G} \,
P_{\sigma}$. The bijection is given such that $x_{\sigma} \in
P_{\sigma}$ if and only if $\rho (x_{\sigma}) = x_{\sigma}
\ot \sigma$, for all $\sigma \in G$.

2. The second example of a Poisson comodule algebra comes from algebraic geometry \cite[Example 2.20]{gue}:
if $V$ is an affine Poisson variety (i.e. the coordinate ring $k[V]$ of $V$ is a Poisson algebra) and $G$ is an algebraic group
acting on $V$ via automorphisms of Poisson varieties, then $k[V]$ is a Poisson $k[G]$-comodule algebra.
\end{examples}

For further details concerning the study of Poisson algebras see \cite{crainic, LPV} and the references therein and for undefined concepts on category theory (resp. Hopf algebras) we refer the reader to \cite{mlane} (resp. \cite{radford, Sw}).

\section{The universal algebra of two Poisson algebras}\selabel{sect2}

Before introducing the main characters of this paper we make the following key observation:
if $P$ is a Poisson algebra and $A$ is a commutative associative algebra then
$P\ot A $ is a Poisson algebra. The associative algebra structure and the Lie bracket are defined  as follows
for all $x$, $y \in P$ and $a$, $b\in A$:
\begin{eqnarray}
(x\ot a) \, (y\ot b) = xy \ot ab, \qquad
\left[ x\ot a, \, y\ot b \right] = \left[x, \, y\right] \ot ab. \eqlabel{2}
\end{eqnarray}
Indeed, having in mind that $A$ is a commutative associative algebra, we have:
\begin{eqnarray*}
&& \left[ x \ot a,\, \underline{(y \ot b)(z \ot c)} \right] \stackrel{\equref{2}}
= \underline{\left[ x \ot a,\, yz \ot bc \right]}  \stackrel{\equref{2}} =  \underline{\left[ x,\, yz\right]} \ot abc\\
&& \stackrel{\equref{1}} =  [x,\,y]\,z \, \ot \,abc + y\, [x,\, z]\, \ot \,abc  \stackrel{\equref{2}} = (\underline{\left[x,\,y\right] \, \ot \,ab)(}z \,\ot\, c) + (y\, \ot \,b)(\underline{\left[x,\, z\right]\, \ot \,ac})\\
&& \stackrel{\equref{2}} = \left[ x\ot a, \, y\ot b \right] (z\, \ot \, c) + (y\, \ot \,b)\left[ x\ot a, \, z\ot c \right]
\end{eqnarray*}
for all $x$, $y$, $z \in P$ and $a$, $b$, $c\in A$,
i.e. \equref{1} holds for $P \ot A$. Furthermore, if $f \colon A \to B$ is an algebra map then ${\rm Id}_{P} \, \ot f :
P \ot A \to P \ot B$ is a morphism of Poisson algebras. To conclude, given a Poisson algebra $P$, assigning $A
\mapsto P \ot A$ defines a functor $P \ot - \, : {\rm ComAlg}_k \to {\rm Poiss}_k$ from the
category of commutative algebras to the category of Poisson algebras. With this remark in hand we can now introduce the following concept:

\begin{definition} \delabel{defcheie}
Let $P$ and $Q$ be two Poisson algebras. The \emph{universal algebra of
$P$ and $Q$} is a pair $\bigl( {\mathcal P} (P, \, Q ), \, \eta \bigl)$ consisting of a commutative algebra
${\mathcal P} (P, \, Q ) \in {\rm ComAlg}_k $ and a Poisson algebra homomorphism
$\eta \colon Q \to P \ot {\mathcal P} (P, \, Q)$ satisfying the following universal property:
for any commutative algebra $A$ and any Poisson algebra homomorphism $g \colon Q \to P \ot A$
there exists a unique algebra homomorphism $\theta \colon {\mathcal P} (P, \, Q ) \to A$ such that the following diagram is commutative:
\begin{eqnarray} \eqlabel{diagrama_univ}
\xymatrix {& Q \ar[r]^-{\eta} \ar[dr]_{g}
& {  P \ot {\mathcal P} (P,
\, Q)} \ar[d]^{ {\rm Id}_{P} \ot \theta }\\
& {} & {P \ot A}}\quad {\rm i.e.}\,\,\, g = \bigl( {\rm Id}_{P} \ot \theta \bigl) \circ \, \eta
\end{eqnarray}
If $Q = P$ then ${\mathcal P} (P) := {\mathcal P} (P, \, P )$ will be called
the \emph{universal coacting bialgebra on $P$}\footnote{The terminology is explained by \thref{univbialg} below.}.
\end{definition}

The universal algebra of two Poisson algebras $P$ and $Q$, if exists, it is unique up to an isomorphism of algebras. In what follows we prove that if $P$
is a finite dimensional Poisson algebra and $Q$ an arbitrary Poisson algebra, then the universal algebra ${\mathcal P} (P, \, Q )$ of $P$ and $Q$ exists and we will provide its explicit construction. We formulate this result in terms of adjoint functors, as the Poisson algebra version of \cite[Theorem 1.1]{Tambara}.

\begin{theorem}\thlabel{adj}
Let $P$ be a finite dimensional Poisson algebra. Then the functor $P \ot - \, \colon {\rm ComAlg}_k \to {\rm
Poiss}_k$ has a left adjoint ${\mathcal P} (P, \, - ) : {\rm Poiss}_k \to {\rm ComAlg}_k$. Furthermore, if $Q$ is an arbitrary Poisson algebra, then
${\mathcal P} (P, \, Q )$ is the universal algebra of $P$ and $Q$.
\end{theorem}

\begin{proof}
Let $n \in \NN^{*}$ be a positive integer and $\{e_1, \cdots, e_n\}$ a basis of the Poisson algebra $P$. We denote by $\{\tau_{i, j}^s \, | \, i, j, s = 1, \cdots, n \}$ and $\{\mu_{i, j}^s \, | \, i, j, s = 1, \cdots, n \}$ the structure constants of $P$ with respect to the associative and Lie structures, i.e. for all
$i$, $j = 1, \cdots, n$ we have:
\begin{eqnarray}
e_i \, e_j = \sum_{s=1}^n \,
\tau_{i, j}^s \, e_s, \qquad
\left[e_i, \, e_j \right]_{P} = \sum_{s=1}^n \,
\mu_{i, j}^s \, e_s. \eqlabel{const3.2}
\end{eqnarray}
We will construct explicitly a left adjoint ${\mathcal P}
(P, \, - ) : {\rm Poiss}_k \to {\rm ComAlg}_k$ for the tensor product functor $P \ot - \, : {\rm ComAlg}_k \to {\rm
Poiss}_k$. To this end, let $Q$ be a Poisson algebra and consider $\{f_i \, | \, i \in I\}$ to be its basis. Then, for all $i$, $j\in I$, we can find two finite subsets $U_{i,j}$ and $V_{i,j}$ of $I$ such that:
\begin{eqnarray}
f_{i}\, f_{j} = \sum_{u \in U_{i, j}}
\, \alpha_{i, j}^u \, f_{u}, \qquad
\left[f_i, \, f_j \right]_{Q} = \sum_{u \in V_{i, j}}
\, \beta_{i, j}^u \, f_{u}\eqlabel{const3.4}
\end{eqnarray}
for some scalars $\alpha_{i, j}^u$, $\beta_{i, j}^u \in k$. Consider now $k [X_{si} \, | \, s = 1, \cdots, n, \, i\in I]$
to be the usual polynomial algebra and let
\begin{equation*}\eqlabel{alguniv}
{\mathcal P} (P, \, Q) :=  k [X_{si} \, | s
= 1, \cdots, n, \, i\in I] / J
\end{equation*}
where $J$ is the ideal generated by all polynomials of the form:
\begin{eqnarray}
\Gamma_{(a, i, j)} ^{(P, \, Q)} = \sum_{u \in
U_{i, j}} \, \alpha_{i, j}^u \, X_{au} - \sum_{s, t = 1}^n \,
\tau_{s, t}^a \, X_{si} X_{tj} \eqlabel{assoc_comp}\\
\Omega_{(a, i, j)} ^{(P, \, Q)} = \sum_{u \in
V_{i, j}} \, \beta_{i, j}^u \, X_{au} - \sum_{s, t = 1}^n \,
\mu_{s, t}^a \, X_{si} X_{tj} \eqlabel{Lie_comp}
\end{eqnarray}
for all $a = 1, \cdots, n$ and $i$, $j\in I$. Denoting $x_{si}
:= \widehat{X_{si}}$, where $\widehat{X_{si}}$ stands for the equivalence class of ${X_{si}}$ in the quotient
algebra ${\mathcal P} (P, \, Q)$, it follows that the relations below hold in
${\mathcal P} (P, \, Q)$:
\begin{eqnarray}
\sum_{u \in
U_{i, j}} \, \alpha_{i, j}^u \, x_{au} = \sum_{s, t = 1}^n \,
\tau_{s, t}^a \, x_{si} x_{tj} \eqlabel{rel1}\\
\sum_{u \in
V_{i, j}} \, \beta_{i, j}^u \, x_{au} = \sum_{s, t = 1}^n \,
\mu_{s, t}^a \, x_{si} x_{tj} \eqlabel{rel2}
\end{eqnarray}
for all $a = 1, \cdots, n$ and $i$, $j \in I$. Next, we consider the following linear map:
\begin{equation}\eqlabel{unitadj}
\eta_{Q} \colon Q \to P \ot {\mathcal P} (P, \, Q), \quad \eta_{Q}
(f_i) := \sum_{s=1}^n \, e_s \ot x_{si},  \quad {\rm for\,\,
all}\,\, i\in I.
\end{equation}
We will see that $\eta_{Q} $ is in fact a Poisson algebra map; indeed, for all $i$, $j\in I$ we have:
\begin{eqnarray*}
&& \left[\eta_{Q} (f_i), \, \eta_{Q} (f_j)
\right]_{P \ot {\mathcal P} (P, \,
Q)} = \left[ \sum_{s=1}^n \, e_s \ot x_{si}, \,\,
\sum_{t=1}^n \, e_t \ot x_{tj} \right]_{P \ot {\mathcal P} (P, \,
Q)}\\
&& = \sum_{s, t =1}^n \, \underline{\left[e_s, \, e_t \right]_{P}}
\ot x_{si} x_{tj} = \sum_{a =1}^n \, e_a \ot \underline{\bigl(\sum_{s, \,t =
1}^n \, \mu_{s, t}^a \, x_{si} x_{tj} \bigl)} \,\, \stackrel{\equref{rel2}}
= \,\, \sum_{a=1}^n \, e_a \ot \bigl( \sum_{u \in V_{i, j}} \,
\beta_{i, j}^u \, x_{au}\bigl) \\
&&= \sum_{u \in V_{i, j}} \, \beta_{i, j}^u \, \eta_{Q}
(f_u)  = \eta_{Q} (\left[f_i, \, f_j \right]_{Q})
\end{eqnarray*}
and
\begin{eqnarray*}
&&\eta_{Q} (f_i)\, \eta_{Q} (f_j)
 = \bigl(\sum_{s=1}^n \, e_s \ot x_{si}\bigl) \,\bigl(
\sum_{t=1}^n \, e_t \ot x_{tj}\bigl)= \sum_{s, t =1}^n \, \underline{e_s \, e_t}
\ot x_{si} x_{tj}\\
&& = \sum_{a =1}^n \, e_a \ot \underline{\bigl(\sum_{s, \,t =
1}^n \, \tau_{s, t}^a \, x_{si} x_{tj} \bigl)} \,\, \stackrel{\equref{rel1}}
= \,\, \sum_{a=1}^n \, e_a \ot \bigl( \sum_{u \in U_{i, j}} \,
\alpha_{i, j}^u \, x_{au}\bigl) = \sum_{u \in U_{i, j}} \, \alpha_{i, j}^u \, \eta_{Q}
(f_u) \\
&& = \eta_{Q} (f_i \, f_j)
\end{eqnarray*}
This shows that $\eta_{Q}$ is indeed a Poisson algebra homomorphism, as claimed. The next step of the proof consists in
showing that for any Poisson algebra $Q$ and any
commutative algebra $A$ the map defined below is bijective:
\begin{equation}\eqlabel{adjp}
\gamma_{Q, \, A} \colon {\rm Hom}_{\rm Alg_k} \, (
{\mathcal P} (P, \, Q), \, A) \to {\rm
Hom}_{\rm Poiss_k} \, (Q, \, P \ot A), \quad
\gamma_{Q, \, A} (\theta) = \bigl( {\rm
Id}_{P} \ot \theta \bigl) \circ \eta_{Q}
\end{equation}
To this end, let $g \colon Q \to P\ot A$ be a
Poisson algebra homomorphism. We have to prove that there exists a unique algebra homomorphism $\theta \colon {\mathcal P} (P, \, Q) \to A$ such that
$g = \bigl( {\rm Id}_{P} \ot \theta \bigl) \circ \, \eta_{Q}$. Let $\{d_{si} \, | \, s = 1, \cdots, n, i\in I \}$ be a family of elements of $A$ such that for
all $i\in I$ we have:
\begin{equation}\eqlabel{constfmor}
g( f_i) = \sum_{s=1}^n \, e_s \ot d_{si}
\end{equation}
Furthermore, as $g \colon Q \to P\ot A$ is a
Poisson algebra map, we can easily conclude that the following compatibilities hold for all $a = 1, \cdots, n$ and $i$, $j \in I$:
\begin{eqnarray}
\sum_{u \in
U_{i, j}} \, \alpha_{i, j}^u \, d_{au} = \sum_{s, t = 1}^n \,
\tau_{s, t}^a \, d_{si} d_{tj} \eqlabel{rel3}\\
\sum_{u \in
V_{i, j}} \, \beta_{i, j}^u \, d_{au} = \sum_{s, t = 1}^n \,
\mu_{s, t}^a \, d_{si} d_{tj} \eqlabel{rel4}
\end{eqnarray}
The universal property of the polynomial algebra yields a unique algebra homomorphism $v \colon k [X_{si} \,
| s = 1, \cdots, n, \, i\in I] \to A$ such that $v (X_{si}) =
d_{si}$, for all $s = 1, \cdots, n$ and $i\in I$. Furthermore, we have
$J \subseteq {\rm Ker} (v)$, where $J$ is the ideal generated by
all polynomials listed in \equref{assoc_comp} and \equref{Lie_comp}. Indeed, for all $i$, $j\in I$
and $a = 1, \cdots, n$ we have:
\begin{eqnarray*}
v \bigl( \Gamma_{(a, i, j)} ^{(P, \, Q)}
 \bigl) = v \bigl( \sum_{u \in
U_{i, j}} \, \alpha_{i, j}^u \, X_{au} - \sum_{s, t = 1}^n \,
\tau_{s, t}^a \, X_{si} X_{tj}\bigl) = \underline{\sum_{u \in U_{i, j}} \, \alpha_{i, j}^u \, d_{au} - \sum_{s,
t = 1}^n \, \tau_{s, t}^a \, d_{si} d_{tj}}
\stackrel{\equref{rel3}} = 0\\
v \bigl( \Omega_{(a, i, j)} ^{(P, \, Q)}
 \bigl) = v \bigl( \sum_{u \in
V_{i, j}} \, \beta_{i, j}^u \, X_{au} - \sum_{s, t = 1}^n \,
\mu_{s, t}^a \, X_{si} X_{tj}\bigl) = \underline{\sum_{u \in V_{i, j}} \, \beta_{i, j}^u \, d_{au} - \sum_{s,
t = 1}^n \, \mu_{s, t}^a \, d_{si} d_{tj}}
\stackrel{\equref{rel4}} = 0
\end{eqnarray*}
Thus, there exists a unique algebra homomorphism
$\theta \colon {\mathcal P} (P, \, Q) \to A$ such
that $\theta (x_{si}) = d_{si}$, for all $s = 1, \cdots, n$ and
$i\in I$. We are left to show that $g = \bigl( {\rm Id}_{P} \ot \theta \bigl) \circ \, \eta_{Q}$. To this end, for all $i\in I$ we have:
\begin{eqnarray*}
\bigl( {\rm Id}_{P} \ot \theta \bigl) \circ \,
\eta_{Q} (f_i) = \bigl( {\rm Id}_{P} \ot
\theta \bigl) \bigl( \sum_{s=1}^n \, e_s \ot x_{si} \bigl) = \sum_{s=1}^n \, e_s \ot d_{si} \stackrel{\equref{constfmor}}
= g (f_i),
\end{eqnarray*}
as desired. We are left to show that $\theta$ is the unique morphism with this property. Indeed, consider $\tilde{\theta} \colon {\mathcal P}
(P, \, Q) \to A$ to be another algebra homomorphism such that $\bigl( {\rm Id}_{P} \ot \tilde{\theta} \bigl) \circ \,
\eta_{Q} (f_i) = g
(f_i)$, for all $i\in I$. Then,  $\sum_{s=1}^n \, e_s \ot
\tilde{\theta} (x_{si}) = \sum_{s=1}^n \, e_s \ot d_{si}$, and
hence $\tilde{\theta} (x_{si}) = d_{si} = \theta (x_{si})$, for
all $s= 1, \cdots, n$ and $i\in I$. As the set $\{x_{si} \, | s= 1, \cdots, n, i \in I \, \}$ generates the algebra ${\mathcal P} (P,
\, Q)$ we can conclude that $\tilde{\theta} =
\theta$. To summarize, we proved that the map
$\gamma_{Q, \, A}$ given by \equref{adjp} is bijective.

The only thing left to show is that given a finite dimensional Poisson algebra $P$, assigning $Q \mapsto {\mathcal P} (P, \,
Q)$ defines a functor  ${\mathcal P} (P, \,
- ) \colon {\rm Poiss}_k \to {\rm ComAlg}_k$. Indeed, let $u \colon
Q_{1} \to Q_2$ be a Poisson
algebra homomorphism. Applying the bijectivity of the map defined by
\equref{adjp} for the Poisson algebra homomorphism $\eta_{Q_2} \circ u$, yields a unique algebra homomorphism $\theta \colon {\mathcal P} (P, \,
Q_1) \to {\mathcal P} (P, \,
Q_2)$ such that:
\begin{eqnarray} \eqlabel{diag2}
\bigl( {\rm Id}_{P} \ot \theta
\bigl) \circ \, \eta_{Q_1} = \eta_{Q_2}
\circ u
\end{eqnarray}
By considering ${\mathcal P}
(P, \, u)$ to be this unique morphism $\theta$, the functor ${\mathcal P}
(P, \, - )$ is fully defined. Moreover, it
can now be easily checked that ${\mathcal P} (P, \, -
)$ is indeed a functor and that $\gamma_{Q, \, A}$ is natural in both variables. Therefore, the functor ${\mathcal P}
(P, \, - )$ is the left adjoint of the functor $P \ot -$. Finally, the bijectivity of the map \equref{adjp} shows that the pair
$\bigl( {\mathcal P} (P, \, Q ), \, \eta_{Q} \bigl)$ is indeed the universal algebra of $P$ and $Q$.
\end{proof}

\begin{remark} \relabel{Jacobi}
\thref{adj} remains valid if we replace ${\rm Poiss}_k$ by the category ${\rm Poiss}_k^1$ of unital Poisson algebras. If $P$ is a unital finite dimensional Poisson algebra,
then the functor $P \ot - \, \colon {\rm ComAlg}_k \to {\rm Poiss}_k^1 $ has a left adjoint
${\mathcal P}^1 (P, \, - ) \colon {\rm Poiss}_k ^1 \to {\rm ComAlg}_k$ which is constructed as follows. If
$\{e_1, \cdots, e_n\}$ is a basis of the Poisson algebra $P$ such that $e_1 := 1_P$ and $Q$ is a
unital Poisson algebra with basis $\{f_i \, | \, i \in I\}$ such that $f_{i_0} := 1_Q$ then we define
$$
{\mathcal P}^1 (P, \,  Q) := {\mathcal P} (P, \,  Q) / L
$$
where $L$ is the ideal of ${\mathcal P} (P, \,  Q)$ generated by $x_{s i_0} - \delta_{s, 1}$, for all $s = 1, \cdots, n$. These new relations are necessary and sufficient for the map $\eta_{Q} \colon Q \to P \ot {\mathcal P}^1 (P, \, Q)$ defined in \equref{unitadj}
to be unital, i.e. $\eta_{Q} (1_Q) = 1_P \ot 1$. The rest of the proof goes exactly as for \thref{adj}.

Furthermore, \thref{adj} can be generalized to the category of Jacobi algebras by repeating verbatim the above proof.
Recall that a \emph{Jacobi algebra} \cite{am15} is a quadruple $J = (J, \, m_J, \, 1_{J},
\, [-, \, -])$, where $(J, m_J, 1_{J})$ is a unital commutative algebra, $(A,
\, [-, \, -])$ is a Lie algebra such that for all $a$, $b$, $c\in
J $ we have:
\begin{equation}\eqlabel{jac1}
[ab, \, c ] =  a \, [b, \, c] + [a, \, c] \, b  - ab \, [1_A, \,
c]
\end{equation}
We can prove that for any Jacobi algebra $J$ and any commutative algebra $A$, the tensor product $J\ot A$ is a Jacobi algebra with
the structures given by \equref{2}. If we denote by ${\rm Jac}_k$ the category of Jacobi algebras, then for any
finite diminesional Jacobi algebra $J$, the functor $J \ot - \, \colon {\rm ComAlg}_k \to {\rm Jac}_k $ has a left adjoint.
\end{remark}

The universal algebra ${\mathcal P} (P, \, Q )$ of two Poisson algebras $P$ and $Q$ as constructed in \thref{adj} is an important tool for comparing the two Poisson algebras:
the first application shows that the set of all usual algebra maps ${\mathcal P} (P, \, Q ) \to k$ parameterize the space of all
Poisson algebra maps $Q \to P$. Indeed, by considering $A := k$, the bijection described in \equref{adjp} comes down to the following:

\begin{corollary}\colabel{morP}
Let $P$ and $Q$ be two Poisson algebras such
that $P$ is finite dimensional. Then the following map is bijective:
\begin{equation}\eqlabel{adjpois}
\gamma \, : {\rm Hom}_{\rm Alg_k} \, ( {\mathcal P} (P,
\, Q), \, k) \to {\rm Hom}_{\rm Poiss_k} \,
(Q, \, P), \quad \gamma (\theta) := \bigl(
{\rm Id}_{P} \ot \theta \bigl) \circ \eta_{Q}
\end{equation}
\end{corollary}

The next applications of the universal algebra ${\mathcal P} (P, \, Q )$ are more nuanced and refer to representations (i.e. Poisson modules) of the two Poisson algebras $P$ and $Q$. In the sequel, we will use the explicit description through generators and relations of the algebra ${\mathcal P} (P, \, Q )$ provided in the proof of \thref{adj}.

\begin{theorem}\thlabel{tensor_f}
Let $P$ and $Q$ be Poisson algebras such that $P$ is finite dimensional, $A = {\mathcal P}(P, \, Q)$ the corresponding universal algebra,
$(U, \blacktriangleright,\curvearrowright) \in {}_P{\mathcal P} \Mm$ a Poisson $P$-module and $(V, \cdot) \in {}_A \Mm$ an $A$-module.

Then $(U \otimes V,\, \triangleright,\, \rightharpoonup) \in {}_Q{\mathcal P} \Mm$ is a Poisson $Q$-module where the actions of $Q$ on $U \otimes V$ are given for all $i \in I$, $l \in U$ and $t \in V$ by:
\begin{eqnarray}
&& f_{i} \triangleright (l \ot t) = \sum_{j=1}^{n}\, (e_{j} \blacktriangleright\, l) \ot (x_{ji} \cdot t)\eqlabel{0.0.1}\\
&& f_{i} \rightharpoonup (l \ot t) = \sum_{j=1}^{n}\, (e_{j} \curvearrowright  l) \ot (x_{ji} \cdot t)\eqlabel{0.0.2}
\end{eqnarray}
In particular, any fixed $(U, \blacktriangleright,\curvearrowright) \in {}_P{\mathcal P} \Mm$ yields a functor $U \ot - \colon {}_A \Mm \to {}_Q{\mathcal P} \Mm$ from the category of $A$-modules to the category of Poisson $Q$-modules; similarly, any fixed $(V, \cdot) \in {}_A \Mm$ gives rise to a functor $ - \ot V \colon {}_P{\mathcal P} \Mm \to {}_Q{\mathcal P} \Mm$ connecting the categories of Poisson modules over $P$ and $Q$.
\end{theorem}

\begin{proof}
We start by showing that $(U \otimes V,\, \triangleright)$ is a left $Q$-module. To thie end, we have:
\begin{eqnarray*}
&& (\underline{f_{i} f_{j}}) \triangleright (l \ot t) \stackrel{\equref{const3.4}}= \sum_{u \in U_{i,j}} \alpha^{u}_{i,j} \underline{f_{u} \triangleright  (l \ot t)} \stackrel{\equref{0.0.1}}= \sum_{u \in U_{i,j}, r=\overline{1,n}}
\,  (\alpha^{u}_{i,j} e_{r} \blacktriangleright\, l) \ot (x_{ru} \cdot t)\\
&&= \sum_{r=1}^{n} \, (e_{r} \blacktriangleright\, l) \ot \bigl(\underline{\sum_{u \in U_{i,j}} \alpha^{u}_{i,j} x_{ru}} \bigl)\, \cdot \,t \stackrel{\equref{rel1}}= \sum_{r,s,p=1}^{n} \tau_{s, p}^r \, (e_{r}  \blacktriangleright\, l) \ot  (x_{si} x_{pj}) \cdot t \\
&&= \sum_{s,p=1}^{n}\bigl(\underline{\sum_{r=1}^{n} \tau_{s, p}^r \, e_{r}} \bigl)  \blacktriangleright\, l \ot  (x_{si} x_{pj}) \cdot t  \stackrel{\equref{const3.2}}=  \sum_{s,p=1}^{n}\, \underline{(e_{s}e_{p}) \blacktriangleright\, l} \, \ot \, (x_{si} x_{pj}) \cdot t \\
&& = \sum_{s,p=1}^{n}\, e_{s} \blacktriangleright (e_{p}\blacktriangleright l)  \ot \underline{ (x_{si} x_{pj}) \cdot t} =  \sum_{p=1}^{n}\, \Bigl(\underline{\sum_{s=1}^{n}\,e_{s} \blacktriangleright (e_{p}\blacktriangleright l)  \ot  x_{si}\cdot( x_{pj} \cdot t)}\Bigl)\\
&&  \stackrel{\equref{0.0.1}}= f_{i} \triangleright \underline{\sum_{p=1}^{n}\, e_{p}\blacktriangleright l  \ot x_{pj} \cdot t}  \stackrel{\equref{0.0.1}}= f_{i} \triangleright \bigl(f_{j} \triangleright (l \ot t) \bigl)
\end{eqnarray*}
We point out that $(U \otimes V,\, \rightharpoonup)$ being a left Lie $Q$-module can be proved exactly as in (the proof of) \cite[Theorem 2.1]{aa2022}. The proof will be finished once we prove that compatibilities \equref{Pmod1} and \equref{Pmod2} hold for $(U \otimes V,\, \triangleright,\, \rightharpoonup)$. Indeed, as compatibilities \equref{Pmod1} and \equref{Pmod2} hold for $(U, \blacktriangleright,\curvearrowright)$ and $A$ is a commutative algebra, for all $i$, $j \in I$ and $l \in U$, $t \in V$, we have:
\begin{eqnarray*}
&& (\underline{f_{i} f_{j}}) \rightharpoonup (l \ot t) \stackrel{\equref{const3.4}}= \sum_{u \in U_{i,j}} \alpha^{u}_{i,j} \underline{f_{u} \rightharpoonup  (l \ot t)} \stackrel{\equref{0.0.2}}= \sum_{u \in U_{i,j}, r=\overline{1,n}}
\,  (\alpha^{u}_{i,j} e_{r} \curvearrowright l) \ot (x_{ru} \cdot t)\\
&&= \sum_{r=1}^{n} \, (e_{r} \curvearrowright l) \ot \bigl(\underline{\sum_{u \in U_{i,j}} \alpha^{u}_{i,j} x_{ru}} \bigl)\, \cdot \,t \stackrel{\equref{rel1}}= \sum_{r,s,p=1}^{n} \tau_{s, p}^r \, (e_{r} \curvearrowright l) \ot  (x_{si} x_{pj}) \cdot t \\
&&= \sum_{s,p=1}^{n}\bigl(\underline{\sum_{r=1}^{n} \tau_{s, p}^r \, e_{r}} \bigl) \curvearrowright l \ot  (x_{si} x_{pj}) \cdot t  \stackrel{\equref{const3.2}}=  \sum_{s,p=1}^{n}\, \underline{(e_{s}e_{p}) \curvearrowright l} \, \ot \, (x_{si} x_{pj}) \cdot t \\
&&\stackrel{\equref{Pmod1}}= \sum_{s,p=1}^{n}\, \bigl(e_{s} \blacktriangleright (e_{p} \curvearrowright l) + e_{p} \blacktriangleright (e_{s} \curvearrowright l) \bigl) \ot  (x_{si} x_{pj}) \cdot t\\
&& = \underline{\sum_{s,p=1}^{n}\, e_{s} \blacktriangleright (e_{p} \curvearrowright l) \ot  x_{si} \cdot (x_{pj} \cdot t)} + \underline{\sum_{s,p=1}^{n}\, e_{p} \blacktriangleright (e_{s} \curvearrowright l) \, \ot \,  x_{pj} \cdot (x_{si} \cdot t)}\\
&& \stackrel{\equref{0.0.1}}=  f_{i} \triangleright \underline{\sum_{p=1}^{n}\, (e_{p} \curvearrowright l) \, \ot \,  (x_{pj} \cdot t)} + f_{j} \triangleright \underline{\sum_{s=1}^{n}\, (e_{s} \curvearrowright l) \ot (x_{si} \cdot  t)}\\
&&\stackrel{\equref{0.0.2}}=   f_{i} \triangleright \bigl(f_{j} \rightharpoonup (l \ot t) \bigl) \, + \, f_{j} \triangleright \bigl(f_{i} \rightharpoonup (l \ot t) \bigl)
\end{eqnarray*}
and
\begin{eqnarray*}
&& \underline{[f_{i},\, f_{j}]} \triangleright (l \ot t) \stackrel{\equref{const3.4}} = \sum_{v \in V_{i,j}} \beta^{u}_{i,j}\, \underline{f_{u} \triangleright  (l \ot t)} \stackrel{\equref{0.0.1}}= \sum_{u \in V_{i,j}, r=\overline{1,n}}
\,  \beta^{u}_{i,j} (e_{r} \blacktriangleright l) \ot (x_{ru} \cdot t)\\
&&= \sum_{r=1}^{n} \, (e_{r} \blacktriangleright l) \ot \bigl(\underline{\sum_{u \in V_{i,j}} \beta^{u}_{i,j} x_{ru}} \bigl)\,\cdot \,t \stackrel{\equref{rel2}}= \sum_{r,s,p=1}^{n} \mu_{s, p}^r \, (e_{r} \blacktriangleright l) \ot  (x_{si} x_{pj}) \cdot t \\
&&= \sum_{s,p=1}^{n}\bigl(\underline{\sum_{r=1}^{n} \mu_{s, p}^r \, e_{r}} \bigl)\, \blacktriangleright \,l \, \ot \, (x_{si} x_{pj}) \cdot t  \stackrel{\equref{const3.2}}=  \sum_{s,p=1}^{n}\, \underline{[e_{s},\,e_{p}] \blacktriangleright l} \, \ot \, (x_{si} x_{pj}) \cdot t \\
&&\stackrel{\equref{Pmod2}}= \sum_{s,p=1}^{n}\, \bigl(e_{s}  \curvearrowright (e_{p} \blacktriangleright l) - e_{p} \curvearrowright (e_{s}
\blacktriangleright l) \bigl) \, \ot \,   (x_{si} x_{pj}) \cdot t\\
&&= \underline{\sum_{s,p=1}^{n}\, \bigl(e_{s} \curvearrowright (e_{p} \blacktriangleright l)\bigl)\, \ot \, x_{si} \cdot (x_{pj} \cdot t)} -  \underline{\sum_{s,p=1}^{n}\, \bigl(e_{p} \curvearrowright (e_{s} \blacktriangleright l)\bigl)\, \ot \, x_{pj} \cdot (x_{si} \cdot t)}\\
&& \stackrel{\equref{0.0.2}}= f_{i} \rightharpoonup \underline{\sum_{p=1}^{n}\, (e_{p} \blacktriangleright l) \ot (x_{pj} \cdot t)} - f_{j} \rightharpoonup \underline{\sum_{s=1}^{n}\, (e_{s} \blacktriangleright l) \ot (x_{si} \cdot  t)}\\
&&\stackrel{\equref{0.0.1}}= f_{i} \rightharpoonup \bigl(f_{j} \triangleright (l \ot t) \bigl) \, - \, f_{j} \rightharpoonup  \bigl(f_{i} \triangleright (l \ot t) \bigl)
\end{eqnarray*}
which concludes the proof.
\end{proof}

Furhermore, if $(U, \blacktriangleright, \curvearrowright) \in {}_P{\mathcal P} \Mm$ is finite dimensional then the first functor constructed
in \thref{tensor_f} admits a left adjoint:

\begin{theorem}\thlabel{adjoint_tens}
Let $P$ and $Q$ be Poisson algebras such that $P$ is finite dimensional, $A = {\mathcal P}(P, \, Q)$ and $(U, \blacktriangleright,\curvearrowright) \in {}_P{\mathcal P} \Mm$ a finite dimensional Poisson $P$-module. Then the functor $U \ot - \colon {}_A \Mm \to {}_Q{\mathcal P} \Mm$ has a left adjoint
${\mathcal U} (U, \, - ) \colon {}_Q{\mathcal P} \Mm \to {}_A \Mm$.
\end{theorem}

\begin{proof}
Let $\{u_1, \cdots, u_m\}$, $m \in \NN^{*}$, be a $k$-basis of the Poisson $P$-module $U$ and denote by $\gamma^{t}_{i,j}$, $\omega^{t}_{i,j} \in k$ the structure constants of $U$ with respect the two module structures, i.e.  for all $i = 1,
\cdots, n$, $j = 1, \cdots, m$ we have:
\begin{equation}\eqlabel{const0.1}
e_i \,  \blacktriangleright \, u_{j}= \sum_{s=1}^m \,
\gamma_{i, j}^s \, u_s, \qquad e_i\, \curvearrowright\, u_{j} = \sum_{s=1}^n \,
\omega_{i, j}^s \, u_s
\end{equation}
where  $\{e_1, \cdots, e_n\}$ is a $k$-basis of $P$. The left adjoint
${\mathcal U} (U, \, - ) \colon {}_Q{\mathcal P} \Mm \to {}_A \Mm$ of the tensor product functor $U \ot -$ will be constructed as follows. First, consider $(V, \vdash,\looparrowright) \in {}_Q{\mathcal P} \Mm$ and $\{v_{r} ~|~ r \in J\}$ its $k$-basis. For all $j \in I$ and $r \in J$ we can find two finite subsets $W_{j,r}$ and $T_{j,r}$ of $J$ such that:
\begin{equation}\eqlabel{const0.2}
f_j \,  \vdash \, v_{r}= \sum_{t \in W_{j,r}} \,
\sigma_{j, r}^t \, v_t, \qquad f_j\, \looparrowright\, v_{r} = \sum_{l \in T_{j,r}} \,
\eta_{j, r}^l \, v_l
\end{equation}
where $\sigma_{j, r}^t$, $\eta_{j, r}^l \in k$ for all $j \in I$, $r \in J$, $t \in W_{j,r}$ and $l \in T_{j,r}$ (recall that
$\{f_i \, | \, i \in I\}$ is a $k$-basis in $Q$). Consider now $\overline{{\mathcal U}}(U, V)$ to be the free $A$-module generated
by the set $\{Y_{i j} ~|~ i = 1, \cdots, m,\, j \in J\}$ and denote by ${\mathcal U}(U, V)$ the quotient of
$\overline{{\mathcal U}}(U, V)$ by its $A$-submodule generated by the following elements:
\begin{eqnarray}
\sum_{p \in
W_{j, i}} \, \sigma_{j, i}^p \, Y_{sp} - \sum_{t = 1}^{m}\sum_{r= 1}^{n}\,
\gamma_{r, t}^s \, x_{r j} \diamond Y_{t i} \eqlabel{0.3}\\
\sum_{p \in
T_{j, i}} \, \eta_{j, i}^p \, Y_{sp} - \sum_{t = 1}^{m}\sum_{r= 1}^{n}\,
\omega_{r, t}^s \, x_{r j} \diamond Y_{t i} \eqlabel{0.4}
\end{eqnarray}
for all $s = 1, \cdots, m$, $i \in J$ and $j \in I$, where $\diamond$ denotes the $A$-module action on $\overline{{\mathcal U}}(U, V)$.

Denoting $y_{tj} := \widehat{Y_{tj}}$, where $\widehat{Y_{tj}}$ stands for the equivalence class of ${Y_{tj}}$ in the quotient
module ${\mathcal U} (U, \, V)$, it follows that the relations below hold in the $A$-module ${\mathcal U} (U, \, V)$:
\begin{eqnarray}
\sum_{p \in
W_{j, i}} \, \sigma_{j, i}^p \, y_{sp} = \sum_{t = 1}^{m}\sum_{r= 1}^{n}\,\gamma_{r, t}^s \, x_{r j} \diamond y_{t i} \eqlabel{0.5}\\
\sum_{p \in
T_{j, i}} \, \eta_{j, i}^p \, y_{sp} = \sum_{t = 1}^{m}\sum_{r= 1}^{n}\,
\omega_{r, t}^s \, x_{r j} \diamond y_{t i} \eqlabel{0.6}
\end{eqnarray}
for all $s = 1, \cdots, m$, $i \in J$ and $j \in I$. Consider now the following linear map:
\begin{equation}\eqlabel{0.7}
\rho_{V} \colon V \to U \ot {\mathcal U} (U, \, V), \quad \rho_{V}
(v_r) := \sum_{s=1}^m \, u_s \ot y_{sr},  \quad {\rm for\,\,
all}\,\, r\in J.
\end{equation}

Note that $\rho_{V}$ is a Poisson $Q$-module map; indeed, for all $j \in I$ and $i \in J$ we have:
\begin{eqnarray*}
&&\rho_{V}(\underline{f_{j} \vdash \, v_{i}}) \stackrel{\equref{const0.2}}= \rho_{V}\bigl(\sum_{p \in W_{j,i}} \sigma_{ji}^{p}\, v_{p} \bigl) = \sum_{p \in W_{j,i}}\sum_{s = 1}^{m}\, \sigma_{ji}^{p}\,u_{s} \ot y_{sp} = \sum_{s=1}^{m} \bigl(u_{s} \ot \underline{\sum_{p \in W_{j,i}} \sigma_{ji}^{p}\, y_{sp}}\bigl)\\
&&\hspace*{-9mm} \stackrel{\equref{0.5}} =  \sum_{s, t = 1}^{m} \sum_{r= 1}^{n}\, \gamma_{r, t}^s\, u_{s} \ot
  x_{r j} \diamond y_{t i} =  \sum_{t = 1}^{m} \sum_{r= 1}^{n}\, \underline{\bigl(\sum_{s=1}^{m}\gamma_{r, t}^s\, u_{s}\bigl)} \ot
  x_{r j} \diamond y_{t i} \stackrel{\equref{const0.1}} =  \sum_{t=1}^{m} \underline{\sum_{r=1}^{n}  e_r \,  \blacktriangleright \, u_{t} \ot
  x_{r j} \diamond y_{t i}}\\
&&\hspace*{-9mm} \stackrel{\equref{0.0.1}}  = \sum_{t=1}^{m} f_{j} \triangleright (u_{t} \ot y_{ti}) =  f_{j}  \triangleright \sum_{t=1}^{m} u_{t} \ot y_{ti} \stackrel{\equref{0.7}}= f_{j}  \triangleright \rho_{V}(v_{i})
\end{eqnarray*}
and
\begin{eqnarray*}
&&\rho_{V}(\underline{f_{j} \looparrowright \, v_{i}}) \stackrel{\equref{const0.2}}= \rho_{V}\bigl(\sum_{p \in T_{j,i}} \eta_{ji}^{p}\, v_{p} \bigl) = \sum_{p \in T_{j,i}}\sum_{s = 1}^{m}\, \eta_{ji}^{p}\,u_{s} \ot y_{sp} = \sum_{s=1}^{m}\bigl(u_{s} \ot \underline{\sum_{p \in T_{j,i}} \eta_{ji}^{p}\, y_{sp}}\bigl)\\
&&\hspace*{-12mm} \stackrel{\equref{0.6}} =  \sum_{s, t = 1}^{m}\sum_{r= 1}^{n}\, \omega_{r, t}^s\, u_{s} \ot
  x_{r j} \diamond y_{t i} =  \sum_{t = 1}^{m}\sum_{r= 1}^{n}\, \underline{\bigl(\sum_{s=1}^{m}\omega_{r, t}^s\, u_{s}\bigl)} \ot
  x_{r j} \diamond y_{t i} \stackrel{\equref{const0.1}} =  \sum_{t=1}^{m} \underline{\sum_{r=1}^{n}  e_r \, \curvearrowright \, u_{t} \ot
  x_{r j} \diamond y_{t i}} \\
&&\hspace*{-12mm} \stackrel{\equref{0.0.2}}  = \sum_{t=1}^{m} f_{j} \rightharpoonup (u_{t} \ot y_{ti}) =  f_{j} \rightharpoonup \sum_{t=1}^{m} u_{t} \ot y_{ti} \stackrel{\equref{0.7}}= f_{j} \rightharpoonup\rho_{V}(v_{i})
\end{eqnarray*}
which concludes our last claim. We can now define for all Poisson $Q$-modules $V$ and all $A$-modules $X$, a bijection between
${\rm Hom}_{{}_A\mathcal{M}} \, \bigl(
{\mathcal U} (U, \, V), \, X\bigl)$ and ${\rm Hom}_{{}_Q{\mathcal P} \Mm} \, (V, \, U \ot X)$ as follows:
\begin{eqnarray}
\Gamma_{V, X} \colon {\rm Hom}_{{}_A\mathcal{M}} \, (
{\mathcal U} (U, \, V), \, X) \to {\rm
Hom}_{{}_Q{\mathcal P} \Mm} \, (V, \, U \ot X), \,\,  \Gamma_{V, X}(\theta) := ({\rm Id}_{U} \ot \theta) \circ \rho_{V} \eqlabel{gammaadj}
\end{eqnarray}
for all $A$-module morphisms $\theta \colon {\mathcal U}(U, V) \to X$. To this end, let $g \colon V \to U\ot X$ be a Poisson $Q$-module
map; we need to find a unique $A$-module map $\theta \colon {\mathcal U} (U, \, V) \to X$ such that
$g = \bigl( {\rm Id}_{U} \ot \theta \bigl) \circ \, \rho_{V}$. Let $\{z_{sr} \, | \, s = 1, \cdots, m, r\in J \}$ be a family of elements of
$X$ such that for all $r\in J$ we have:
\begin{equation}\eqlabel{constfmor}
g(v_r) = \sum_{s=1}^m \, u_s \ot z_{sr}.
\end{equation}
Furthermore, as $g \colon V \to U\ot X$ is a Poisson $Q$-modules
map, we can easily prove that the following compatibilities hold for all $s = 1, \cdots, m$, $i \in J$ and $j \in I$:
\begin{eqnarray}
\sum_{p \in
W_{j, i}} \, \sigma_{j, i}^p \, z_{sp} = \sum_{t = 1}^{m}\sum_{r= 1}^{n}\,
\gamma_{r, t}^s \, x_{r j} \cdot z_{t i}  \eqlabel{rel0.0.3}\\
\sum_{p \in
T_{j, i}} \, \eta_{j, i}^p \, z_{sp} = \sum_{t = 1}^{m}\sum_{r= 1}^{n}\,
\omega_{r, t}^s \, x_{r j} \cdot z_{t i} \eqlabel{rel0.0.4}
\end{eqnarray}
where $\cdot$ denotes the $A$-module action on $X$. The universal property of the free module yields a unique $A$-module map $\overline{\theta} \colon \overline{{\mathcal U}}(U, V) \to X$ such that $\overline{\theta}(Y_{sr}) = z_{sr}$, for all $s = 1, \cdots, m$ and $r\in J$. Moreover, ${\rm Ker} (\overline{\theta})$ contains the $A$-submodule of $\overline{{\mathcal U}}(U, V)$ generated by the elements listed in \equref{0.3} and \equref{0.4}. Indeed, as $\theta \colon {\mathcal U}(U, V) \to X$ is a morphism of $A$-modules we have:
\begin{eqnarray*}
\overline{\theta}\bigl(\sum_{p \in
W_{j, i}} \, \sigma_{j, i}^p \, Y_{sp} - \sum_{t = 1}^{m}\sum_{r= 1}^{n}\,
\gamma_{r, t}^s \, x_{r j} \diamond Y_{t i} \bigl) = \underline{\sum_{p \in
W_{j, i}} \, \sigma_{j, i}^p \, z_{sp} - \sum_{t = 1}^{m}\sum_{r= 1}^{n}\,
\gamma_{r, t}^s \, x_{r j} \cdot z_{t i}}
\stackrel{\equref{rel0.0.3}} = 0\\
\overline{\theta}\bigl(\sum_{p \in
T_{j, i}} \, \eta_{j, i}^p \, Y_{sp} - \sum_{t = 1}^{m}\sum_{r= 1}^{n}\,
\omega_{r, t}^s \, x_{r j} \diamond Y_{t i}\bigl)) = \underline{\sum_{p \in
T_{j, i}} \, \eta_{j, i}^p \, z_{sp} - \sum_{t = 1}^{m}\sum_{r= 1}^{n}\,
\omega_{r, t}^s \, x_{r j} \cdot z_{t i}} \stackrel{\equref{rel0.0.4}} = 0
\end{eqnarray*}
for all $s = 1, \cdots, m$, $i \in J$ and $j \in I$. This shows that there exists a unique $A$-module map $\theta \colon {\mathcal U}(U, V) \to X$ such that $\theta(y_{sr}) = z_{sr}$, for all $s = 1, \cdots, m$ and $r\in J$. Furthermore, this implies that for all $r \in J$ we have:
\begin{eqnarray*}
\bigl({\rm Id}_{U} \ot \theta \bigl) \circ \,
\rho_{V} (v_r) = \bigl({\rm Id}_{U} \ot
\theta \bigl) \bigl( \sum_{s=1}^{m} \, u_s \ot y_{sr} \bigl) = \sum_{s=1}^{m} \, u_s \ot z_{sr} \stackrel{\equref{constfmor}}
= g (v_{r})
\end{eqnarray*}
$\theta$ is obviously unique with this property and therefore the map
$\Gamma_{V, X}$ is bijective.

We are left to show that given a finite dimensional Poisson $P$-module $U$, assigning $V \mapsto {\mathcal U} (U, \,
V)$ defines a functor  ${\mathcal U}
(U, \, - ) \colon {}_Q{\mathcal P} \Mm \to {}_A \Mm$. Indeed, let $h \colon
V_{1} \to V_2$ be a Poisson $Q$-modules map. The bijectivity of $\Gamma_{V_{1},\, {\mathcal U} (U, \,
V_2)}$ applied for the Poisson $Q$-modules map $\rho_{V_2} \circ h \colon V_{1} \to U \ot \,{\mathcal U} (U, \,
V_2)$, yields a unique $A$-module map $\overline{h} \colon {\mathcal U} (U, \,
V_1) \to {\mathcal U} (U, \,
V_2)$ such that:
$$
\bigl( {\rm Id}_{U} \ot \overline{h}
\bigl) \circ \, \rho_{V_1} = \rho_{V_2}
\circ h
$$
By setting ${\mathcal U}
(U, \, h)$ to be this unique morphism $\overline{h}$, the functor ${\mathcal U}
(U, \, - )$ is fully defined. Moreover, it
can now be easily checked that ${\mathcal U} (U, \, -
)$ is indeed a functor and that $\Gamma_{V, \, X}$ is natural in both variables. Therefore, ${\mathcal U}
(U, \, - )$ is the left adjoint of the functor $U \ot -$.
\end{proof}

Keeping the notations and the assumptions of \thref{tensor_f} we can prove the following:

\begin{theorem}\thlabel{adjoint_tens_2}
Let $P$ and $Q$ be two Poisson algebras such that $P$ is finite dimensional, $A = {\mathcal P}(P, \, Q)$ and let $V = (V, \cdot)$ be a finite dimensional $A$-module.
Then the functor $ - \ot V \colon {}_P{\mathcal P} \Mm \to {}_Q{\mathcal P} \Mm$ has a left adjoint
${\mathcal V} (V, \, - ) \colon {}_Q{\mathcal P} \Mm \to {}_P{\mathcal P} \Mm$.
\end{theorem}

\begin{proof} Since the proof goes in the same manner as the one of \thref{adjoint_tens}, we only indicate its main steps.
Let $\{v_1, \cdots, v_m\}$, $m \in \NN^{*}$, be a $k$-basis of the $A$-module $V$ and denote by $\gamma_{i,j,s}^{t} \in k$ the structure constants of $V$, i.e.  for all $i = 1,
\cdots, n$, $j\in J$ and $s = 1, \cdots, m$ we have:
\begin{equation*}\eqlabel{const0.1.bb}
x_{ij} \,  \cdot \, v_{s}= \sum_{t=1}^m \,
\gamma_{i, j, s}^t \, v_t
\end{equation*}
Let  $(W, \vdash, \looparrowright) \in {}_Q{\mathcal P} \Mm$ be a Poisson $Q$-module and $\{w_{r} ~|~ r \in J\}$ its $k$-basis. For all $j \in I$ and $r \in J$ we can find two finite subsets $S_{j,r}$ and $T_{j,r}$ of $J$ such that:
\begin{equation*}\eqlabel{const0.2.bb}
f_j \,  \vdash \, w_{r}= \sum_{t \in S_{j,r}} \,
\sigma_{j, r}^t \, w_t, \qquad f_j\, \looparrowright\, w_{r} = \sum_{s \in T_{j,r}} \,
\eta_{j, r}^s \, w_s
\end{equation*}
where $\sigma_{j, r}^t$, $\eta_{j, r}^s \in k$, for all $j \in I$, $r \in J$, $t \in S_{j,r}$ and $s \in T_{j,r}$. Using \reref{freeposs} we can now define ${\mathcal V} (V, \, W ) = \bigl( {\mathcal V} (V, \, W ), \, \blacktriangleright, \, \curvearrowright  \bigl)$ as the free Poisson $P$-module generated by the set $\{y_{j i} ~|~  j \in J, \, i = 1, \cdots, m \}$ subject to the following relations:
\begin{eqnarray}
&& \sum_{t \in S_{j,r}} \, \sigma_{j, r}^t \, y_{ra} =
\sum_{i = 1}^{n}\sum_{b = 1}^{m} \, \gamma_{i, j, b}^a \, (e_i \blacktriangleright y_{rb} ) \eqlabel{rel_noulmod1} \\
&& \sum_{s \in T_{j, r}} \, \eta_{j, r}^s \, y_{sa} = \sum_{b = 1}^{m} \sum_{i = 1}^{n} \,
\gamma_{i, j, b}^a \,\, (e_i \curvearrowright y_{rb}) \eqlabel{rel_noulmod2}
\end{eqnarray}
for all $j\in I$, $r\in J$ and $a = 1, \cdots, m$. Now relations \equref{rel_noulmod1} and \equref{rel_noulmod2} allow us to easily prove that the linear map
defined for any $r\in J$ by:
$$
\eta_ W : W \to {\mathcal V} (V, \, W ) \ot V, \qquad \eta_ W (w_r) := \sum_{s=1}^m \, y_{rs} \ot v_s
$$
is a morphism of Poisson $Q$-modules and, analogous to the proof of \thref{adjoint_tens}, the canonical map
$$
{\rm Hom}_{{}_P{\mathcal P} \Mm} \, ( {\mathcal V} (V, \, W), \, U) \to {\rm Hom}_{{}_Q{\mathcal P} \Mm} \, (W, \, U \ot V), \qquad \theta \mapsto (\theta \ot {\rm Id}_{V}) \circ \eta_{W}
$$
is a natural isomorphism for any Poisson $Q$-module $W$ and any Poisson $P$-module $U$. The proof is now finished.
\end{proof}

Before giving some examples, it will be useful to observe the following: since the bracket on the Lie algebras on $P$ and $Q$
is skew-symmetric we have $\mu_{i, i}^s = \beta_{i,i}^u = 0$, $\mu_{i, j}^s = - \mu_{j, i}^s$ and $\beta_{i,j}^u = - \beta_{j,i}^u$. Consequently, relations \equref{rel2} are automatically fulfilled for $i = j$.

\begin{examples} \exlabel{exalguniv}
1. Let $P$ and $Q$ be two Poisson algebras such that $P$ is finite dimensional and the associative algebra structures on both $P$ and $Q$ are the trivial ones (i.e.
$xy := 0$, for any $x$, $y \in P$ (resp. $Q$)). Thus $P$ and $Q$ are just Lie algebras viewed as Poisson algebras. Then, ${\mathcal P}(P, \, Q)$
is exactly the universal algebra ${\mathcal A}(P, \, Q)$ of the two Lie algebras as constructed in \cite[Theorem 2.1]{am20}. In particular, if the Lie algebras structures
on $P$ and $Q$ are also the abelian ones, then ${\mathcal P}(P, \, Q) \cong k [X_{si} \, | s = 1, \cdots, n, \, i\in I]$, where $n = {\rm dim}_k (P)$ and $|I| = {\rm dim}_k (Q)$.

In general, ${\mathcal P}(P, \, Q)$ is the quotient of the universal algebra ${\mathcal A}(P, \, Q)$ of the two Lie algebras $P$ and $Q$, through the ideal generated by the relations listed in  \equref{rel1}.

2. Let $P :=k$ be the $1$-dimensional Poisson algebra, i.e. the constant structures are $\tau_{1,1}^1 = 1$ and $\mu_{1,1}^1 = 0$. For any Poisson algebra
$Q$ with a $k$-basis $\{f_i \, | \, i \in I\}$ and the constant structures $\alpha_{i, j}^u$, $\beta_{i, j}^u \in k$ given by \equref{const3.4}, the universal algebra
${\mathcal P}(k, \, Q)$ is the algebra generated by the commuting variables $x_{i}$, $i \in I$, subject to the relations for any $i$, $j \in I$:
\begin{eqnarray*}
\sum_{u \in U_{i, j}} \, \alpha_{i, j}^u \, x_{u} = x_{i} x_{j}, \qquad
\sum_{u \in V_{i, j}} \, \beta_{i, j}^u \, x_{u} = 0. \eqlabel{rel0.2prima}
\end{eqnarray*}
The other way around, let $Q := k$ and $P$ an $n$-dimensional Poisson algebra with the constant structures $\{\tau_{i, j}^s \, | \, i, j, s = 1, \cdots, n \}$ and $\{\mu_{i, j}^s \, | \, i, j, s = 1, \cdots, n \}$ given by \equref{const3.2}. Then the universal algebra ${\mathcal P}(P, \, k)$ is the algebra generated by the commuting variables
$x_1, \cdots, x_n$ subject to the relations:
\begin{eqnarray*}
\sum_{s, t = 1}^n \, \tau_{s, t}^a \, x_{s} x_{t} = x_a, \qquad
\sum_{s, t = 1}^n \, \mu_{s, t}^a \, x_{s} x_{t} = 0
\end{eqnarray*}
for all $a = 1, \cdots, n$.

3. Let $k$ be a field of characteristic $\neq 2$, $P := k[X]/(X^2)$ viewed as a Poisson algebra with the abelian bracket and $Q := {\rm aff} (2, k)$ the affine
$2$-dimensional Lie algebra with basis $\{f_1, \, f_2\}$ and bracket given by $\left[f_1, \, f_2 \right] = f_2$ viewed as a Poisson algebra
with the trivial multiplication ($x y := 0$, for all $x$, $y \in Q$). Then:
\begin{eqnarray*}
{\mathcal P} \bigl( P, \, Q \bigl) &\cong& \, k [X_{11}, X_{12},
X_{21}, X_{22}]/(X_{11}^2, \, X_{12}, \, X_{11}X_{21}, \, X_{22}) \\
& \cong& \, k[X, Y]/(X^2, XY)
\end{eqnarray*}
Indeed, the only non-zero structure constants of $P$ and $Q$ are: $\tau_{1,1}^1 = \tau_{1,2}^2 = \tau_{2,1}^1 = 1$ and
$\beta_{1,2}^2 = 1 = - \beta_{2,1}^2$. A direct computation shows that, among the sixteen compatibilities resulting from the defining relations \equref{rel1} and \equref{rel2} of ${\mathcal P} \bigl( P, \, Q \bigl)$, after eliminating the redundant relations the only remaining ones are the following:
$x_{11}^2 = 0$, $x_{12} = 0$, $2\, x_{11}x_{21} = 0$ and $x_{22} = 0$. The conclusion now follows.
\end{examples}

\section{The universal coacting bialgebra on a finite dimensional Poisson algebra. Applications} \selabel{sect3}

Let $P$ be a finite dimensional Poisson algebra having $\{e_1, \cdots, e_n\}$ as a $k$-basis. The description of the commutative
algebra ${\mathcal P} (P) := {\mathcal P} (P, \, P) $ given by \thref{adj} is the following: if $\{\tau_{i, j}^s \, | \, i, j, s = 1, \cdots, n \}$ and
$\{\mu_{i, j}^s \, | \, i, j, s = 1, \cdots, n \}$ are the structure constants of $P$ with respect to the associative and
Lie structures as given by \equref{const3.2}, then ${\mathcal P} (P)$ is the free commutative algebra generated by $\{x_{si} \, | \, s, i
= 1, \cdots, n, \}$ and the relations:
\begin{equation}\eqlabel{relatii2}
\sum_{u = 1}^n  \, \tau_{i, j}^u \, x_{au} = \sum_{s, t = 1}^n \,
\tau_{s, t}^a \, x_{si} x_{tj}, \qquad
\sum_{u = 1}^n \, \mu_{i, j}^u \, x_{au} = \sum_{s, t = 1}^n \, \mu_{s, t}^a \, x_{si} x_{tj}
\end{equation}
for all $a$, $i$, $j = 1, \cdots, n$. Furthermore, the map
\begin{equation}\eqlabel{unitadj2}
\eta_{P} : P \to P \ot {\mathcal P} (P), \quad \eta_{P} (e_i) := \sum_{s=1}^n
\, e_s \ot x_{si}, \quad {\rm for\,\, all}\,\, i = 1, \cdots, n
\end{equation}
is a Poisson algebra homomorphism. By considering $Q := P$ in the bijection described in \equref{adjp}
we obtain:

\begin{corollary}\colabel{initialobj}
Let $P$ be a finite dimensional Poisson algebra. Then for any comutative algebra $A$
and any Poisson algebra homomorphism $f : P \to P \ot A$, there exists a
unique algebra homomorphism $\theta: {\mathcal P} (P) \to A$ such that $f =
({\rm Id}_{P} \ot \theta) \circ \eta_{P}$.
\end{corollary}

Next we show that the commutative algebra ${\mathcal P} (P)$ can be endowed with a bialgebra structure such that
$(P, \, \eta_{P})$ becomes a right Poisson ${\mathcal P} (P)$-comodule algebra.

\begin{proposition} \prlabel{bialgebra}
Let $P$ be a Poisson algebra of dimension $n$. Then
there exists a unique bialgebra structure on ${\mathcal P} (P)$ such that the Poisson algebra homomorphism
$\eta_{P} : P \to P \ot {\mathcal P} (P)$ becomes a right ${\mathcal P} (P)$-comodule structure on $P$. The comultiplication and the counit on ${\mathcal P}
(P)$ are given by
\begin{equation} \eqlabel{deltaeps}
\Delta (x_{ij}) = \sum_{s=1}^n \, x_{is} \ot x_{sj} \quad {\rm
and} \quad  \varepsilon (x_{ij}) = \delta_{i, j}
\end{equation}
for all $i$, $j=1, \cdots, n$.
\end{proposition}

\begin{proof}
Consider the Poisson algebra homomorphism $(\eta_{P} \ot {\rm
Id}_{{\mathcal P} (P)} ) \, \circ \,
\eta_{P} \colon P \to
P \ot {\mathcal P} (P) \ot {\mathcal P}
(P)$. \coref{initialobj} yields a unique algebra homomorphism $\Delta :
{\mathcal P} (P) \to {\mathcal P} (P) \ot
{\mathcal P} (P)$ such that the following holds:
\begin{eqnarray} \eqlabel{delta}
({\rm Id}_{P}
\ot \Delta) \circ \eta_{P} = (\eta_{P} \ot {\rm
Id}_{{\mathcal P} (P)} ) \, \circ \,
\eta_{P}.
\end{eqnarray}
Applying \equref{delta} for each $e_i$, $i = 1, \cdots, n$ and using \equref{unitadj2} we obtain the following:
\begin{eqnarray*}
&& \sum_{t=1}^n \, e_t \ot \Delta (x_{ti}) = (\eta_{P}
\ot {\rm Id}) (\sum_{s=1}^n \, e_s \ot x_{si}) = \sum_{s=1}^n (
\sum_{t=1}^n \, e_t \ot x_{ts}) \ot x_{si}\\
&& = \sum_{t=1}^n \, e_t \ot (\sum_{s=1}^n x_{ts} \ot x_{si} )
\end{eqnarray*}
which comes down to $\Delta (x_{ti}) = \sum_{s=1}^n \, x_{ts} \ot x_{si}$,
for all $t$, $i=1, \cdots, n$. Note that $\Delta$ is obviously coassociative. In a similar fashion,
applying once again \coref{initialobj}, we obtain a unique algebra homomorphism $\varepsilon\colon
{\mathcal P} (P) \to k$ such that the following holds:
\begin{eqnarray} \eqlabel{epsilo}
({\rm Id}_{P} \ot \varepsilon) \circ \eta_{P} = {\rm can}
\end{eqnarray}
where ${\rm can} : P \to P
\ot k$ is the canonical isomorphism, ${\rm can} (x) = x \ot 1$,
for all $x\in P$. Applying \equref{epsilo} for each $e_i$, $i = 1, \cdots, n$, we obtain $\varepsilon
(x_{ij}) = \delta_{i, j}$, for all $i$, $j=1, \cdots, n$. It can be easily checked that
$\varepsilon$ is a counit for $\Delta$, and therefore ${\mathcal P} (P)$ is a bialgebra. Furthermore, \equref{delta} and \equref{epsilo} imply that the canonical map $\eta_{P} \colon P \to P
\ot {\mathcal P} (P)$ defines a right ${\mathcal P} (P)$-comodule structure on $P$.
\end{proof}

The key property of ${\mathcal P} (P)$ is the following Poisson algebra version of \cite[Theorem 2.11]{am20}:

\begin{theorem}\thlabel{univbialg}
Let $P$ be a finite dimensional Poisson algebra. Then, $({\mathcal P} (P), \, \eta_{P})$ is the initial object of the
category ${\rm CoactBialg}_P$ of all commutative bialgebras coacting on $P$ and we call it the \emph{universal coacting bialgebra of} $P$.
\end{theorem}

\begin{proof} The statement of the theorem comes down to showing that for any commutative bialgebra $B$ and any Poisson algebra
homomorphism $f \colon P \to P \otimes B$
which makes $P$ into a right $B$-comodule there exists
a unique bialgebra homomorphism $\theta \colon {\mathcal P}
(P) \to B$ such that the following diagram is
commutative:
\begin{eqnarray} \eqlabel{univbialg}
\xymatrix {& P \ar[r]^-{\eta_{P}} \ar[dr]_-{f
} & {P \ot {\mathcal P} (P )} \ar[d]^{ {\rm Id}_{P} \ot \theta }\\
& {}  & {P \ot B} }
\end{eqnarray}
To start with, using \coref{initialobj}, we obtain a unique algebra homomorphism $\theta
\colon {\mathcal P} (P) \to B$ such that diagram~\equref{univbialg} commutes. The proof will be finished
once we show that $\theta$ is a coalgebra homomorphism as well.
This follows by using again \coref{initialobj}. Indeed, we obtain a unique algebra
homomorphism $\psi \colon {\mathcal P} (P) \to B \otimes B$ such that the following holds:
\begin{equation}\eqlabel{101}
({\rm Id}_{P} \otimes \psi) \circ \eta_{P} = \bigl({\rm Id}_{P}\otimes\,
\Delta_{B} \circ \theta\bigl)\circ \eta_{P}
\end{equation}
Obviously the algebra homomorphism $\Delta_{B} \circ \theta \colon {\mathcal P} (P) \to B \otimes B$ fulfills the above compatibility. The proof will be finished once we show that $(\theta \otimes
\theta) \circ \Delta \colon {\mathcal P} (P) \to B \otimes B$ fulfills the same compatibility.
Indeed, as  $f \colon
P \to P \otimes B$ is a right $B$-comodule
structure, we have:
\begin{eqnarray*}
\bigl({\rm Id}_{P} \otimes\, (\theta \otimes \theta) \circ \Delta\bigl)\circ \,\eta_{P} &=& \bigl({\rm Id}_{P} \otimes \theta \otimes \theta \bigl)\circ \underline{\bigl({\rm Id}_{P} \otimes \Delta\bigl)\circ \, \eta_{P}}\\
&\stackrel{\equref{delta}} {=}& \bigl({\rm Id}_{P} \otimes \theta \otimes \theta \bigl)\circ (\eta_{P} \otimes {\rm Id}_{{\mathcal P} (P)})\circ \eta_{P}\\
&=& \bigl(\underline{({\rm Id}_{P} \otimes \theta) \circ \eta_{P}}\ \otimes \theta\bigl)\circ \, \eta_{P}\\\
&\stackrel{\equref{univbialg}} {=}& \bigl(f \otimes \theta\bigl)\circ \, \eta_{P}\\
&=& (f \otimes {\rm Id}_{B})\circ \underline{({\rm Id}_{P} \otimes \theta) \circ \, \eta_{P}}\\
&\stackrel{\equref{univbialg}} {=}& \underline{(f \otimes  {\rm Id}_{B})\circ f}\\
&=& ({\rm Id}_{P} \otimes \Delta_{B}) \circ \underline{f}\\
&\stackrel{\equref{univbialg}} {=}& ({\rm Id}_{P} \otimes \Delta_{B}) \circ ({\rm Id}_{P} \otimes \theta) \circ \eta_{P}\\
&=& ({\rm Id}_{P} \otimes \Delta_{B} \circ \theta) \circ \eta_{P}
\end{eqnarray*}
as desired. Similarly, one can show that  $\varepsilon_B \, \circ \, \theta = \varepsilon$ and the proof is now finished.
\end{proof}

By considering Takeuchi's commutative Hopf envelope \cite{T} of the bialgebra ${\mathcal P} (P)$ we obtain, using \thref{univbialg}, the following:

\begin{corollary} \colabel{unihopf}
Let $P$ be a finite dimensional Poisson algebra. Then the category ${\rm CoactHopf}_P$ consisting of all commutative Hopf algebras coacting on $P$
has an initial object $\bigl({\mathcal H} (P), \, \lambda_{P}\bigl)$ and we call it the \emph{universal coacting Hopf algebra} of $P$.
\end{corollary}

\begin{proof}
Indeed, the forgetful functor $U \colon  {\rm ComHopf}_k \to {\rm ComBiAlg}_k$ from the category of commutative Hopf algebras
to the category of commutative bialgebras has a left adjoint $L \colon {\rm ComBiAlg}_k \to {\rm ComHopf}_k$ (\cite[Theorem 65, (2)]{T}).
If we denote by $\mu \colon 1_{{\rm ComBiAlg}_k} \to UL$ the unit of the adjunction $L \dashv U$, then we can easily prove, in the
spirit of \cite[Theorem 2.13]{am20}, that the pair $\bigl({\mathcal H} (P) := L({\mathcal P} (P)), \, \lambda_{P} := ({\rm Id}_{P} \ot \mu_{{\mathcal P} (P)}) \, \circ \,  \eta_{P}\bigl)$ is the initial object in the category ${\rm CoactHopf}_P$ of all
commutative Hopf algebras coacting on $P$.
\end{proof}

\begin{remark} \relabel{hopfenv}
The dual versions of \thref{univbialg} and \coref{unihopf} regarding the actions of commutative bialgebras (resp. Hopf algebras) on a Poisson algebra also hold.
For a Poisson algebra $P$, we can define the category ${\rm ActBialg}_P$ (resp. ${\rm ActHopf}_P$) of all commutative bialgebras (respectively Hopf algebras) which act on $P$. More precisely, the objects of  ${\rm ActBialg}_P$  (resp. ${\rm ActHopf}_P$) are pairs $(B, \, \mu_P)$ consisting of a commutative bialgebra (resp. Hopf algebra) $B$ and a linear map $\mu_P \colon P \ot B \to P$, such that $(P, \mu_P)$ is a (right) \emph{Poisson $B$-module algebra}, i.e. $\mu_P$ is a (right) $B$-module structure on $P$ as well as a Poisson algebra map. Using the same arguments as in \cite[Theorem 4.14]{AGV2}, we can prove that ${\rm ActBialg}_P$ (resp. ${\rm ActHopf}_P$) has a final object.
\end{remark}

Next we will present two important applications of the bialgebra ${\mathcal P} (P)$. These are the Poisson algebra version of similar results obtained
for Lie/associative algebras in \cite{am20, mil}. First, recall the well known fact that for any bialgebra
$H$, we have $G(H^{\rm o}) = {\rm Hom}_{\rm Alg_k} (H, \,  k)$, the set of all algebra homomorphisms $H\to k$ (see (\cite[pag. 62]{radford})).

\begin{theorem} \thlabel{automorf}
Let $P$ be a finite dimensional Poisson algebra with
basis $\{e_1, \cdots, e_n\}$ and $U\bigl (G\bigl( {\mathcal P}
(P)^{\rm o} \bigl)\bigl)$ the group of all
invertible group-like elements of the finite dual ${\mathcal P}
(P)^{\rm o}$. Then the map defined for any $\theta \in
U\bigl(G\bigl( {\mathcal P} (P)^{\rm o} \bigl)\bigl)$
and $i = 1, \cdots, n$ by:
\begin{equation} \eqlabel{izomono}
\overline{\gamma} : U \bigl(G\bigl( {\mathcal P}
(P)^{\rm o} \bigl) \bigl) \to {\rm Aut}_{{\rm Poiss}}
(P), \qquad \overline{\gamma} (\theta) (e_i) :=
\sum_{s=1}^n \, \theta(x_{si}) \, e_s
\end{equation}
is an isomorphism of groups.
\end{theorem}

\begin{proof}
Using \coref{morP} for $Q:= P$ yields the bijective map
$$
\gamma : {\rm Hom}_{\rm Alg_k} ({\mathcal P} (P) , \,
k) \to {\rm End}_{{\rm Poiss}} (P), \quad \gamma (\theta)
= \bigl( {\rm Id}_{P} \ot \theta \bigl) \circ
\eta_{P}
$$
Furthermore, as discussed above we have $ {\rm Hom}_{\rm Alg_k} ({\mathcal P}
(P) , k) = G\bigl( {\mathcal P} (P)^{\rm o}
\bigl)$ and based on \equref{unitadj2} it follows easily that $\gamma$ takes the form given in \equref{izomono}. We denote by $\overline{\gamma}$ the restriction
of $\gamma$ to the invertible elements of the two monoids where the monoid structure on
${\rm End}_{{\rm Poiss}} (P)$ is given by the usual
composition of endomorphisms while $G\bigl( {\mathcal P} (P)^{\rm o}
\bigl)$ is a monoid with respect to the convolution product, i.e.
\begin{equation}\eqlabel{convolut}
(\theta_1 \star \theta_2) (x_{sj}) = \sum_{t=1}^n \,
\theta_1(x_{st}) \theta_2(x_{tj})
\end{equation}
for all $\theta_1$, $\theta_2 \in G\bigl( {\mathcal P}
(P)^{\rm o} \bigl)$ and $j$, $s = 1, \cdots, n$. Therefore, the proof will be finished by showing that $\gamma$ is a monoid
isomorphism and this can be shown exactly as in \cite[Theorem 3.1]{am20}.
\end{proof}

Next, for a given abelian group $G$, we describe all $G$-gradings on a Poisson algebra $P$.

\begin{proposition}\prlabel{graduari}
Let $G$ be an abelian group and $P$ a finite dimensional Poisson algebra. There exists a bijection between
the set of all $G$-gradings on $P$ and the set of all
bialgebra homomorphisms ${\mathcal P} (P) \to k[G]$ given such that the $G$-grading on $P
= \oplus_{\sigma \in G} \, P_{\sigma}^{(\theta)} $
associated to a bialgebra map $\theta: {\mathcal P} (P)
\to k[G]$ can be described as follows:
\begin{equation}\eqlabel{gradass}
P_{\sigma}^{(\theta)} := \{ x \in  P \, | \,
\bigl({\rm Id}_{P} \ot \theta \bigl) \, \circ \,
\eta_{P} (x) = x \ot \sigma  \}
\end{equation}
for all $\sigma \in G$.
\end{proposition}

\begin{proof} \thref{univbialg} applied for the
commutative bialgebra $B := k[G]$ yields a bijection between the
set of all bialgebra homomorphisms ${\mathcal P} (P)
\to k[G]$ and the set of all Poisson algebra homomorphisms $f
\colon P \to P \otimes k[G]$ which make
$P$ into a right $k[G]$-comodule. The proof is now finished since we have shown in \exref{nouexnev} that the latter set is in
bijective correspondence with the set of all $G$-gradings on the Poisson algebra $P$.
\end{proof}

Our next aim is to classify all $G$-gradings on a Poisson algebra $P$. To this end, we introduce the following:

\begin{definition}\delabel{conjug}
Let $G$ be an abelian group and $P$ a finite
dimensional Poisson algebra. Two homomorphisms of bialgebras $\theta_1,
\theta_2: {\mathcal P} (P) \to k[G]$ are called
\emph{conjugate} and denote this by $\theta_1 \approx \theta_2$, if there
exists $g \in U\bigl (G\bigl( {\mathcal P} (P)^{\rm o}
\bigl)\bigl)$ an invertible group-like element of the finite dual
${\mathcal P} (P)^{\rm o}$ such that $\theta_2 = g
\star \theta_1 \star g^{-1}$, in the convolution algebra ${\rm
Hom} \bigl( {\mathcal P} (P) , \, k[G] \bigl)$.
\end{definition}

Throughout, ${\rm Hom}_{\rm BiAlg} \, \bigl( {\mathcal P}
(P) , \, k[G] \bigl)/\approx $ will denote the quotient set of the
set of all bialgebra homomorphisms ${\mathcal P} (P) \to k[G]$
by the above equivalence relation and let $\hat{\theta}$ denote the
equivalence class of $\theta \in {\rm Hom}_{\rm BiAlg} \, \bigl(
{\mathcal P} (P) , \, k[G] \bigl)$. The next theorem
classifies all $G$-gradings on a Poisson algebra $P$.

\begin{theorem} \thlabel{nouaclas}
Let $G$ be an abelian group, $P$ a finite dimensional
Poisson algebra and $G$-${\rm \textbf{gradings}}(P)$ the
set of isomorphism classes of all $G$-gradings on $P$. Then the map
$$
{\rm Hom}_{\rm BiAlg} \, \bigl( {\mathcal P} (P) , \,
k[G] \bigl)/\approx  \,\,\, \mapsto \,\, G{\rm-\textbf{gradings}}
(P), \qquad \hat{\theta} \mapsto
P^{(\theta)} := \oplus_{\sigma \in G} \,
P_{\sigma}^{(\theta)}
$$
where $P_{\sigma}^{(\theta)} = \{ x \in P
\, | \, \bigl({\rm Id}_{P} \ot \theta \bigl) \, \circ
\, \eta_{P} (x) = x \ot \sigma  \}$, for all $\sigma
\in G$, is bijective.
\end{theorem}

\begin{proof}
Since the associative and Lie/Leibniz algebra counterparts of this result have been proved in detail in \cite[Theorem 3.4]{mil} and \cite[Theorem 3.5]{am20}, respectively, we will be brief. First, note that by \prref{graduari}, for any $G$-grading $P
= \oplus_{\sigma \in G} \, P_{\sigma}$ there exists a unique bialgebra homomorphism $\theta:
{\mathcal P} (P) \to k[G]$ such that
$P_{\sigma} = P_{\sigma}^{(\theta)}$, for all $\sigma \in G$. The proof will be finished once we show that any two $G$-gradings on $P$, say $P^{(\theta_1)}$ and $P^{(\theta_2)}$, associated to two bialgebra homomorphisms $\theta_1$,
$\theta_2 \colon {\mathcal P} (P) \to k[G]$, are isomorphic
if and only if $\theta_1 \approx  \theta_2$. Indeed, recall from \exref{nouexnev} that defining a $G$-grading on $P$ is
in one-to-one correspondence to defining a right
$k[G]$-comodule structure $\rho \colon P \to P \ot
k[G]$ on $P$ which is also a Poisson algebra homomorphism. Now two $G$-gradings
$P^{(\theta_1)}$ and $P^{(\theta_2)}$ are
isomorphic if and only if $(P, \, \rho^{(\theta_1)})$
and $(P, \, \rho^{(\theta_2)})$ are isomorphic both as algebras and as right $k[G]$-comodules; this comes down to the existence of an automorphism $w \colon P \to P$ of the Poisson algebra
$P$ such that $\rho^{(\theta_2)} \, \circ w = \bigl(w \ot
{\rm Id}_{k[G]}\bigl) \, \circ \rho^{(\theta_1)}$. By \thref{automorf}, for any Poisson algebra
automorphism $w : P \to P$ there exists a unique invertible group-like element of the finite dual $g \in U\bigl
(G\bigl( {\mathcal P} (P)^{\rm o} \bigl)\bigl)$ such that $w = w_g$ is given for any $i =
1, \cdots, n$ by
\begin{equation} \eqlabel{3001}
w_g (e_i) = \sum_{s=1}^n \, g(x_{si}) \, e_s
\end{equation}
where $\{e_1, \cdots, e_n\}$ is a basis in $P$. A straightforward computation shows that the Poisson algebra automorphism $w_g \colon P \to P$ is also a right
$k[G]$-comodule map if and only if the following holds:
\begin{equation} \eqlabel{3002}
\sum_{s=1}^n \, g(x_{as}) \theta_1 (x_{si}) = \sum_{s=1}^n \,
\theta_2 (x_{as}) g(x_{si})
\end{equation}
Having in mind that $\{x_{ai}\}_{a, i = 1, \cdots, n}$ is a system
of generators of ${\mathcal P} (P)$) we can easily conclude that \equref{3002} reduces to $g \star
\theta_1 = \theta_2 \star g$. This finishes the proof as $g\colon {\mathcal P}
(P) \to k$ is an invertible element in the convolution algebra ${\rm Hom} \bigl( {\mathcal P} (P) , \, k[G]
\bigl)$ which shows that  $\theta_1 \approx  \theta_2$.
\end{proof}

We will give now an explicit example which describes the initial object in the category of all commutative
bialgebras that coacts on a certain $3$-dimensional Poisson algebra.

\begin{example}
Let $P$ be the $3$-dimensional Poisson algebra with $k$-basis $\{e_1, \, e_2, \, e_3\}$ and Poisson algebra structure given by
$e_1 ^2 := e_2$, $\left[e_1, \, e_3 \right] := e_3$ (undefined multiplications and brackets are all zero). Then, there exists an isomorphism of bialgebras
\begin{eqnarray*}
{\mathcal P} (P) \, \cong \, k[X, Y, Z, T]/(T - XT)
\end{eqnarray*}
where the latter has the following bialgebra structure:
\begin{eqnarray*}
&& \Delta (\widehat{X}) = \widehat{X} \ot \widehat{X}, \qquad \varepsilon (\widehat{X}) = 1 \\
&& \Delta (\widehat{Y}) = \widehat{Y} \ot \widehat{X} + \widehat{X}^2 \ot \widehat{Y}, \qquad \varepsilon(\widehat{Y}) = 0 \\
&& \Delta (\widehat{Z}) = \widehat{Z} \ot \widehat{X} + \widehat{T} \ot \widehat{Z}, \qquad \varepsilon (\widehat{Z}) = 0 \\
&& \Delta (\widehat{T}) = \widehat{T} \ot \widehat{T}, \qquad \varepsilon (\widehat{T}) = 1
\end{eqnarray*}
The canonical coaction $\eta_{P} : P \to P \ot k[X, Y, Z, T]/(T - XT)$ of this bialgebra on $P$ is given by:
\begin{eqnarray*}
&& \eta_{P} (e_1) = e_1 \ot \widehat{X} + e_2 \ot \widehat{Y} + e_3 \ot \widehat{Z} \\
&& \eta_{P} (e_2) = e_2 \ot \widehat{X}^2, \qquad \eta_{P} (e_3) = e_3 \ot \widehat{T}.
\end{eqnarray*}
Indeed, note first that the only non-zero structure constants of $P$ are: $\tau_{1,1}^2 = 1$ and
$\mu_{1,3}^3 = 1 = - \mu_{3,1}^3$. Now, a careful analysis of the $54$ defining relations of ${\mathcal P}( P)$ arising from \equref{relatii2}, leads to the conclusion that after eliminating the redundant ones, we are left with the following:
\begin{eqnarray*}
&& x_{12} = 0, \quad x_{13} = 0, \quad x_{23} = 0, \quad x_{32} = 0, \quad
x_{22} = x_{11}^2, \quad x_{33} = x_{11} x_{33}.
\end{eqnarray*}
The conclusion now follows by denoting $\widehat{X} = x_{11}$, $\widehat{Y} = x_{21}$, $\widehat{Z} = x_{31}$ and $\widehat{T} = x_{33}$.
\end{example}

\end{document}